\documentclass[12pt]{amsart}
\usepackage{color,psfrag,epsfig,url,hhline
}
\numberwithin{figure}{section}
\numberwithin{table}{section}
\newtheorem{theorem}{Theorem}[section]

\newtheorem{lemma}[theorem]{Lemma}
\newtheorem{prop}[theorem]{Proposition}

\theoremstyle{definition}
\newtheorem{definition}[theorem]{Definition}
\newtheorem{example}[theorem]{Example}

\newtheorem{cor}[theorem]{Corollary}

\theoremstyle{remark}
\newtheorem{remark}[theorem]{Remark}

\numberwithin{equation}{section}
\newfont{\tap}{tap scaled 650}

\def \N{{\mathbb N}}
\def \H{{\mathbb H}}
\def \E{{\mathbb E}}
\def \R{{\mathbb R}}

\def \Z{{\mathbb Z}}
\def \S{{\mathbb S}}

\def \Q{{\mathbb Q}}

\def \A{{\mathcal A}}
\def \A{{f}}
\def \[{[ }
\def \]{] }

\definecolor{dgreen}{rgb}{0,0.5,0}        
\definecolor{dred}{rgb}{0.5,0,0}        

\DeclareMathOperator{\arccosh}{arccosh}
\DeclareMathOperator{\tr}{tr}

\begin{document}

\title{Geometry of mutation classes of rank $3$ quivers}
\author{Anna Felikson and Pavel Tumarkin}
\address{Department of Mathematical Sciences, Durham University, Science Laboratories, South Road, Durham, DH1 3LE, UK}
\email{anna.felikson@durham.ac.uk, pavel.tumarkin@durham.ac.uk}
\thanks{AF was partially supported by EPSRC grant EP/N005457/1}



\begin{abstract}
We present a geometric realization for all mutation classes of quivers of rank $3$ with real weights. This realization is via linear reflection groups for acyclic mutation classes and via  groups generated by $\pi$-rotations for the cyclic ones.
The geometric behavior of the model turns out to be controlled by the Markov constant $p^2+q^2+r^2-pqr$, where $p,q,r$ are the elements of exchange matrix. We also classify skew-symmetric mutation-finite real $3\times 3$ matrices and explore the structure of acyclic representatives in finite and infinite mutation classes.  

\end{abstract}

\maketitle
\setcounter{tocdepth}{1}
\tableofcontents

\section{Introduction and main results}

Mutations of quivers were introduced by Fomin and Zelevinsky in~\cite{FZ1} in the context of cluster algebras. Mutations are involutive transformations decomposing the set of quivers into equivalence classes called {\em mutation classes}. 
Knowing the structure of mutation classes gives a lot of information about the corresponding cluster algebras. It is especially beneficial if there exists a certain combinatorial or geometric model for mutations. This is the case, for example, of quivers for cluster algebras originating from bordered marked surfaces~\cite{FG,GSV,FST,FT}, where mutations are modeled by flips of triangulations. Note that such quivers are {\em mutation-finite} (i.e., their mutation classes are finite).    

There is a model for mutations of quivers containing a representative without oriented cycles in their mutation class (such quivers are called {\em mutation-acyclic}): it was shown in~\cite{S2} that mutations of mutation-acyclic quivers can be modeled by reflections of a tuple of positive vectors in a certain quadratic space (we call this a {\it realization by reflections}). One of the goals of this paper is to construct a model for mutations of {\em mutation-cyclic} quivers of rank $3$ (see Section~\ref{s-def} for precise definitions). 

Mutation classes of rank $3$ quivers were studied in~\cite{ABBS,BBH,BFZd,FeSTu,S,W}. In particular, the {\em Markov constant} $C(Q)=p^2+q^2+r^2-pqr$ for a cyclic quiver $Q$ with weights $(p,q,r)$ was introduced in~\cite{BBH}, and $C(Q)$ was proved to be an invariant of a mutation class. We prove that mutations of mutation-cyclic rank $3$ quivers can be modeled by symmetries (or {\em $\pi$-rotations}) of triples of points on a hyperbolic plane. Combining our results with ones of~\cite{BBH} we obtain the following theorem. 

\setcounter{section}{4}
\setcounter{theorem}{4}
\begin{theorem}
Let $Q$ be a rank $3$ quiver with real weights. Then
\begin{itemize}
\item[(1)] if $Q$ is mutation-acyclic then $C(Q)\ge 0$ and  $Q$ admits a realization by reflections; 
\item[(2)] if $Q$ is mutation-cyclic then $C(Q)\le 4$ and  $Q$ admits a realization by $\pi$-rotations; 
\item[(3)] $Q$ admits both realizations (by reflections and by $\pi$-rotations) if and only if
$Q$ is cyclic with $p,q,r\ge 2$ and $C(Q)=4$.

\end{itemize}
\end{theorem}

For an individual mutation-acyclic quiver, the Markov constant also controls the signature of the quadratic space where the mutations are modeled by reflections. The possible signatures for rank $3$ quivers are $(3,0)$, $(2,0,1)$ and $(2,1)$, which can be identified with the sphere $\S^2$, Euclidean plane $\E^2$ and the hyperbolic plane $\H^2$ respectively after considering appropriate projectivization. We prove the following statement.

\setcounter{theorem}{7}
\begin{cor} 
Let $Q$ be a rank $3$ quiver with real weights.
\begin{itemize}
\item[(1)] If $Q$ is acyclic then $C(Q)\ge 0$ and there is a realization by reflections
\begin{itemize}
\item[-] in $\H^2$ if $C(Q)>4$;
\item[-] in $\E^2$ if $C(Q)=4$;
\item[-] in $\S^2$ if $C(Q)<4$.
\end{itemize}

\item[(2)] If $Q=(p,q,r)$ is cyclic with $\min(p,q,r)<2$  or with $r=2$, $p\ne q$, then 
$Q$ is mutation-acyclic,
$C(Q)\ge 0$ and there is a realization by reflections
\begin{itemize}
\item[-] in $\H^2$ if $C(Q)>4$;
\item[-] in $\E^2$ if $C(Q)=4$;
\item[-] in $\S^2$ if $C(Q)<4$.
\end{itemize}

\item[(3)] If $Q=(p,q,r)$ is cyclic with $\min(p,q,r)\ge 2$  then 
\begin{itemize}
\item[-] if $C(Q)>4$ then $Q$ is mutation-acyclic and there is a realization by reflections in $\H^2$; 
\item[-] if $C(Q)<4$ then $Q$ is mutation-cyclic and there is a realization by $\pi$-rotations;
\item[-] if $C(Q)=4$ then $Q$ is mutation-cyclic and there a both realizations. 
\end{itemize}
\end{itemize}
\end{cor}

Throughout the whole paper, we allow a quiver to have real weights, so all the results concern a more general class of quivers than is usually considered. A classification of mutation-finite quivers with integer weights in rank $3$ is extremely simple: there are two quivers in the mutation class of an orientation of $A_3$ Dynkin diagram, two quivers in the mutation class of an acyclic orientation of $A_2^{(1)}$ extended Dynkin diagram, and the Markov quiver. However, in the case of real weights the question is more interesting. We classify all the finite mutation classes in rank $3$ by proving the following theorem.    

\setcounter{section}{6}
\setcounter{theorem}{10}
\begin{theorem}
Let $Q$ be a connected rank $3$ quiver with real weights. Then $Q$ is of finite mutation type if and only if it is mutation-equivalent
to one of the following quivers:
\begin{itemize}
\item[(1)] $(2,2,2)$;
\item[(2)]  $(2\cos (\pi/n ),2\cos(\pi/n),2)$,  $n\in \Z_+$;
\item[(3)] $(1,1,0)$, $(1,\sqrt 2,0)$, $(1,2\cos\pi/5,0)$, $(2\cos\pi/5,2\cos 2\pi/5,0)$, $(1,2\cos2\pi/5,0)$.
\end{itemize}
\end{theorem}

Finally, we observe that the structure of acyclic representatives in mutation classes of quivers with real weights is very different from one in integer case. According to~\cite{CK}, all acyclic representatives in any integer mutation class can be mutated to each other via source-sink mutations only, i.e. by mutations which just reverse directions of arrows incident to the mutation vertex. This is not the case for quivers with real weights: already finite mutation classes may have two essentially distinct acyclic representatives (see Table~\ref{finite}), and infinite mutation classes have infinitely many ones. Moreover, we prove an even stronger statement.

\setcounter{section}{7}
\setcounter{theorem}{1}
\begin{theorem}
Let $Q=(p,q,r)$ be a mutation-acyclic quiver with $0<C(Q)<4$.
Then there exists an acyclic quiver $Q'$ which can be obtained from $Q$ in at most
 $\!\lfloor\pi/\!\arcsin\!\frac{\sqrt{4-C(Q)}}{2}\!\rfloor$ mutations. 

\end{theorem}

\medskip
\noindent
The paper is organized as follows. 

In Section~\ref{s-refl} we recall basic notions on quiver mutations, and then remind the construction of a realization of mutations of mutation-acyclic quivers via reflections. Section~\ref{s-rot} is devoted to a construction of a realization of mutations via $\pi$-rotations, we also show that mutations of every mutation class can be realized via either reflections or $\pi$-rotations. In Section~\ref{Markov-sec} we use the Markov constant to show that all mutation-cyclic mutation classes admit realizations by $\pi$-rotations, and specify the geometry depending on the value of the Markov constant. Section~\ref{s-discrete} is devoted to a discussion of discreteness of the group generated by either reflections or $\pi$-rotations representing mutations. In Section~\ref{s-fin} we classify finite mutation classes of rank $3$ quivers, and in Section~\ref{s-ac} we discuss the structure of acyclic representatives in mutation classes.    

\subsection*{Acknowledgements}
We would like to thank Philipp Lampe and Lutz Hille for stimulating discussions inspiring the current project. We are grateful to John Parker for referring us to the results of~\cite{CJ} and for a concise introduction to vanishing sums of roots of unity. We also thank Arkady Berenstein for sharing with us the results of~\cite{BFZd}.

\setcounter{section}{1}






\section{Mutation-acyclic quivers via reflections}
\label{s-refl}

In this section we show that mutations of  a mutation-acyclic rank $3$ quiver can be modeled via some linear reflection group acting on a sphere $\S^2$, on a Euclidean plane $\E^2$ or on a hyperbolic plane $\H^2$. The results of this section can be deduced from~\cite{BGZ} (see also~\cite{S2,ST} for more general picture), we give a geometric interpretation and observe that taking real weights instead of integer ones does not affect the proofs. 

\subsection{Quiver mutations}
\label{s-def}
First, we remind the basics on quivers and their mutations.

A {\em quiver} $Q$ is a finite oriented graph with weighted edges containing no loops and no $2$-cycles. We allow the weights to be any positive real numbers. We call the directed edges {\em arrows}, while drawing a quiver we omit weights equal to one. By {\em rank} of $Q$ we mean the number of its vertices.

For every vertex $k$ of a quiver $Q$ one can define an involutive operation  $\mu_k$ called {\it mutation of $Q$ in direction $k$}. This operation produces a new quiver  denoted by $\mu_k(Q)$ which can be obtained from $Q$ in the following way (see~\cite{FZ1}): 
\begin{itemize}
\item
orientations of all arrows incident to the vertex $k$ are reversed; 
\item
for every pair of vertices $(i,j)$ such that $Q$ contains arrows directed from $i$ to $k$ and from $k$ to $j$ the weight of the arrow joining $i$ and $j$ changes as described in Figure~\ref{quivermut}.
\end{itemize} 

\begin{figure}[!h]
\begin{center}
\psfrag{a}{\small $p$}
\psfrag{b}{\small $q$}
\psfrag{c}{\small $r$}
\psfrag{d}{\small $r'$}
\psfrag{k}{\small $k$}
\psfrag{mu}{\small $\mu_k$}
\epsfig{file=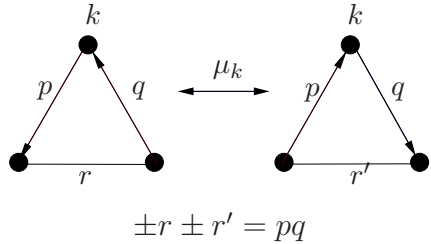,width=0.35\linewidth}\\
\medskip
$\pm{r}\pm{r'}={pq}$
\caption{Mutations of quivers. The sign before ${r}$ (resp., ${r'}$) is positive if the three vertices form an oriented cycle, and negative otherwise. Either $r$ or $r'$ may vanish. If $pq$ is equal to zero then neither the value of $r$ nor orientation of the corresponding arrow changes.}
\label{quivermut}

\end{center}
\end{figure}

Given a quiver $Q$, its {\it mutation class} is a set of all quivers obtained from $Q$ by all sequences of iterated mutations. All quivers from one mutation class are called {\it mutation-equivalent}. A quiver is called {\em minimal} if the sum of its weights is minimal amongst the whole mutation class. 

Quivers without loops and $2$-cycles are in one-to-one correspondence with real skew-symmetric matrices $B=\{b_{ij}\}$, where $b_{ij}>0$ if and only if there is an arrow from $i$-th vertex to $j$-th one with weight $b_{ij}$. In terms of the matrix $B$ the mutation $\mu_k$ can be written as $\mu_k(B)=B'$, where
$$b'_{ij}=\left\{
           \begin{array}{ll}
             -b_{ij}, & \hbox{ if } i=k \hbox{ or } j=k; \\
             b_{ij}+\frac{|b_{ik}|b_{kj}+b_{ik}|b_{kj}|}{2}, & \hbox{ otherwise.}\\
           \end{array}
         \right.
$$
This transformation is called a {\it matrix mutation}.

A rank $3$ quiver (and the corresponding $3\times 3$ matrix) is called {\em cyclic} if its arrows compose an oriented cycle, and is called {\em acyclic} otherwise. A quiver (and the matrix) is {\em mutation-cyclic} if all representatives of the mutation class are cyclic, and {\em mutation-acyclic} otherwise.

\newpage

\subsection{Construction}
\subsubsection{The initial configuration.}
Let $Q$ be an acyclic rank $3$ quiver and let $B$ be the corresponding skew-symmetric $3\times 3$ matrix (we will assume $b_{ij}\ne 0$). Consider a symmetric matrix with non-positive off-diagonal entries $M(B)=(m_{ij})$, where 
$$ m_{ii}=2, \qquad \quad  m_{ij}=-|b_{ij}| \ \text{ if } i\ne j. $$
This matrix defines a quadratic form and we may consider $M(B)$ as a Gram matrix (i.e., the matrix of inner products) of some triple of vectors $(v_1,v_2,v_3)$  in a quadratic space $V$ of the same signature as $M(B)$ has. Considering the projectivization $P(V)=V/\R_+$, the images $\Pi_i$ of the hyperplanes $\pi_{i}=v_i^\perp$ in $P(V)$ define lines in a space $X$ of constant curvature, where $X$ is the sphere $\S^2$ if $M(B)$ is positive definite, a Euclidean plane $\E^2$ if $M(B)$ is degenerate positive semidefinite, or $\H^2$ if $M(B)$ is of signature $(2,1)$. The scalar product $(v_i,v_j)$ characterizes the mutual position of the corresponding lines:
$$  
|(v_i,v_j)|=\begin{cases}  
2\cos\angle (\Pi_i,\Pi_j)<2 & \text{ if  $\Pi_i$ intersects $\Pi_j$,}\\ 
2&  \text{ if  $\Pi_i$ is parallel to $\Pi_j$,}\\
2\cosh d(\Pi_i,\Pi_j)>2  & \text{ otherwise,} \end{cases}
$$
where $d(\Pi_i,\Pi_j)$ is the distance between diverging planes in $\H^2$.

Consider also the halfplanes
$$\Pi_i^-=\{u\in P(V)\ | \ (u,v_i)<0  \}.$$
Let $F=\Pi_1^-\cap \Pi_2^-\cap \Pi_3^-$ be the intersection of these half-planes. 
Since $(v_i,v_j)\le 0$, $F$ is an acute-angled domain (i.e. $F$ has no obtuse angles). 

\subsubsection{Reflection group.}
Given a vector $v_i\in V$  with $(v_i,v_i)=2$ one can consider a {\it reflection } 
$$r_i(u)=u-(u,v_i)v_i$$ 
with respect to $\Pi_i=v_i^\perp$.
It is straightforward to see that $r_i$ preserves the scalar product in $V$ 
(and hence, acts on $X$ as an isometry) 
and that $r_i(v_i)=-v_i$, i.e. that $r_i$ is an isometry of $X$ preserving $\Pi_i$ and interchanging the half-spaces into which $X$ is decomposed by $\Pi_i$.

We denote by $G$ the group generated by reflections $r_1,r_2,r_3$.  

\subsubsection{Mutation.}
The initial acyclic quiver $Q$ (and the initial matrix $B$) corresponds to the initial set of generating reflections in the group $G$ and to the initial domain $F\subset P(V)$.
Applying mutations, we will obtain other sets of generating reflections in $G$ as well as other domains in $P(V)$.

More precisely, define mutation of the set of generating reflections by partial conjugation:
$$
\mu_k(r_j)=\begin{cases}  
r_kr_jr_k & \text{if  $b_{jk}>0$,}\\ 
r_j&  \text{otherwise.}\end{cases}
$$
Consequently, the mutation of the triple of vectors (and of the triple of lines) is defined by partial reflection:
$$
\mu_k(v_j)=\begin{cases}  
v_j-(v_j,v_k)v_k & \text{if  $b_{jk}>0$,}\\ 
-v_k & \text{if $j=k$,}\\
v_j&  \text{otherwise.}\end{cases}
$$

\begin{example}
\label{ex-3}
Let $Q$ be a rank $3$ acyclic quiver (corresponding to a matrix with $b_{ij}>0$ for $i<j$) 
and let $\{v_1,v_2,v_3\}$ be the corresponding vectors in $V$ (see Fig.~\ref{3-ac}). 
We will assume that $V$ is an indefinite space (so, $P(V)$ contains a projective model of $\H^2$),
the points of $\H^2$ are the ones inside the circle.
Consider the mutation $\mu_2$: it reflects $v_1$, changes the direction of $v_1$ and preserves $v_3$.

\begin{figure}[!h]
\begin{center}
\psfrag{p}{\scriptsize \color{dred} $p$}
\psfrag{q}{\scriptsize \color{dred}  $q$}
\psfrag{r}{\scriptsize  \color{dred} $r$}
\psfrag{pq}{\scriptsize \color{dred}  $pq+r$}
\psfrag{1}{\scriptsize $1$}
\psfrag{2}{\scriptsize $2$}
\psfrag{3}{\scriptsize $3$}
\psfrag{p.}{\scriptsize $p$}
\psfrag{q.}{\scriptsize $q$}
\psfrag{r.}{\scriptsize $r$}
\psfrag{pq.}{\scriptsize $pq+r$}
\psfrag{1_}{\color{dgreen} \scriptsize $v_1$}
\psfrag{2_}{\color{dgreen}\scriptsize $v_2$}
\psfrag{3_}{\color{dgreen}\scriptsize $v_3$}
\psfrag{1'}{\color{dgreen}\scriptsize $v_1'$}
\psfrag{2'}{\color{dgreen}\scriptsize $v_2'$}
\psfrag{3'}{\color{dgreen}\scriptsize $v_3'$}
\psfrag{t}{\scriptsize $\mu_2$}
\epsfig{file=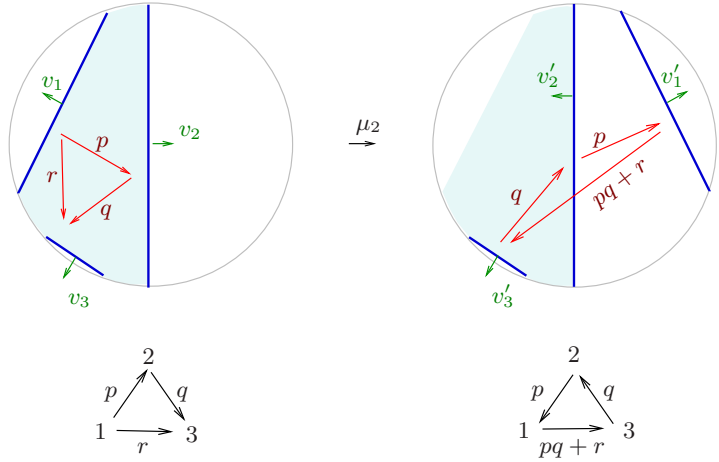,width=0.59\linewidth}
\caption{Mutation of the triple of lines agreeing with the mutation of a quiver (see Fig.~\ref{quivermut})}
\label{3-ac}
\end{center}
\end{figure}

\end{example}

\begin{remark}
\label{signed-mut}
The mutation of a configuration of lines as defined above is not an involution. 
This can be fixed as follows:
choose any vector $u\in V$ and define mutation $\mu_k$ 
\begin{itemize}
\item[-] as above (i.e. by reflection of $v_j$ if  $b_{jk}>0$)  for the cases when $(u,v_k)<0$;
\item[-] by reflection of $v_j$ if  $b_{jk}<0$ for the cases when  $(u,v_k)>0$.
\end{itemize}
Note that the configurations of lines obtained as a result of application of two versions of the definition differ by reflection in $v_k$ only. Throughout the paper we will  only use the configurations up to conjugation by an element of $G$, so it will be sufficient for us to use the initial definition.

\end{remark}

\subsection{Geometric realization by reflections}

\begin{lemma}[\cite{BGZ}, Corollary of Proposition 3.2]
\label{refl}
Let $Q$ be a rank $3$ quiver, and let $B$ be the corresponding skew-symmetric matrix. 
Let $V=\langle v_1,v_2,v_3\rangle $ be a quadratic space and suppose that
\begin{itemize}

\item[(1)] $(v_i,v_i)=2$ for $i=1,2,3$;
\item[(2)] $|(v_i,v_j)|=|b_{ij}|$ for $1\le i<j\le 3$;
\item[(3)] if $(v_i,v_j)\ne 0$ for all $i\ne j$, then the number of pairs $(i,j)$ such that  $(v_i,v_j)>0$ is even if $Q$ is acyclic and odd if $Q$ is cyclic. 
\end{itemize}

Then the set of vectors ${\bf v'}=(\mu_k(v_1),\mu_k(v_2),\mu_k(v_3))$ satisfies conditions (1)--(3) for $B'=\mu_k(B)$.

\end{lemma}

The statement of the lemma is proved in~\cite{BGZ} for integer skew-symmetrizable matrices, however, their proof works for real skew-symmetric matrices as well. One can also note that for any skew-symmetric matrix $B$ there exists a quadratic three-dimensional space $V$ and a triple of vectors $v_1,v_2,v_3\in V$ satisfying assumptions of the lemma.

\begin{definition}
\label{def-realization-refl}
Let $B$ be a $3\times 3$ skew-symmetric matrix.
We say that a tuple of vectors ${\bf v}=(v_1,v_2,v_3)$ is a {\it geometric realization by reflections} of $B$ if conditions (1)-(3)
of Lemma~\ref{refl} are satisfied. We also say that $\bf v$ provides a {\it realization} of the mutation class of $B$ if the mutations of 
$\bf v$ via partial reflections agree with the mutations of $B$, i.e. if conditions (1)--(3) are satisfied in every seed.

Given a geometric realization  $(v_1,v_2,v_3)$ of $B$, we consider the lines $l_i=\{u \   | \ (u,v_i)=0 \}$. The (unordered) 
triple of lines $(l_1,l_2,l_3)$  
will be also called a {\it geometric realization by reflections} of $B$.
(This definition makes sense as properties (1)-(3) do not depend on the choice of vectors orthogonal to  $(l_1,l_2,l_3)$: changing the sign of a vector changes signs of precisely two inner products, so property (3) stays unaffected.)
A realization of $B$ will be also called a realization of the corresponding quiver $Q$.
\end{definition}

\begin{cor}
\label{cor: ac by relf}
Every acyclic mutation class has a geometric realization by reflections.

\end{cor}

\begin{proof}
In view of Lemma~\ref{refl} it is sufficient to find a geometric realization for an acyclic seed.
This is provided by the construction above (notice that for the initial acyclic seed we get $(v_i,v_j)<0$, so condition (3) holds).

\end{proof}


\begin{remark}
\label{acute}
In contrast to quivers with integer weights, mutation classes of quivers with real weights may have more than one acyclic representative (modulo sink-source mutations), we postpone the discussion of this fact till the last section. Meanwhile, we would like to make one observation.

As we have mentioned above, a triple of lines corresponding to an initial acyclic quiver determines an acute-angled domain. In fact, this holds for any acyclic quiver in the mutation class as well: this immediately follows from Property (3) of Lemma~\ref{refl}.

Moreover, the same Property (3) implies the converse: if a triple of lines determines an acute-angled domain, then it cannot represent a cyclic quiver. Thus, acyclic quivers in the mutation class are exactly those represented by acute-angled configurations. 

\end{remark}

\section{Mutation-cyclic quivers via $\pi$-rotations}
\label{s-rot}
Similarly to acyclic mutation classes realized by partial reflections in $\S^2$, $\E^2$ or $\H^2$, we will use $\pi$-rotations in $\H^2$ to build a geometric realization for mutation-cyclic classes. 

\newpage
\subsection{Construction}

\subsubsection{The initial configuration.}
Let $Q$ be a cyclic rank $3$ quiver and let $B$ be the corresponding  skew-symmetric  $3\times 3$ matrix (we will assume $b_{12},b_{23},b_{31}>0$). 
We will also assume $|b_{ij}|\ge 2$ for all $i\ne j$ (in view of Lemma~\ref{r<2} below this is always the case for quivers in mutation-cyclic classes).

Let $V$ be a quadratic space of signature $(2,1)$ and suppose that $v_1,v_2,v_3$ are negative vectors with
$$
(v_i,v_i)=-2, \qquad |(v_i,v_j)|=|b_{ij}| \text{ for $i\ne j$}.
$$
Geometrically, each of $v_i$ corresponds to some point  in the hyperbolic plane $\H^2$, the scalar product $(v_i,v_j)$ represents the distance $d(v_i,v_j)$ between the corresponding points:
$$
(v_i,v_j)=-2\cosh d(v_i,v_j).
$$  

It is not immediately evident that for every mutation-cyclic matrix $B$ there is a corresponding triple of vectors $v_1,v_2, v_3$, we will prove this in Section~\ref{Markov-sec}.

\subsubsection{$\pi$-rotations group.}
With every point $x\in \H^2$ (i.e. with every negative vector $v\in V$, $(v,v)<0$) we can associate a rotation by $\pi$ around $x$
(also called point symmetry, or point reflection, or a central symmetry): it is an isometry which preserves the point $x$, takes every line through $x$ to itself and interchanges the rays from $x$ on the line. It is easy to check that a $\pi$-rotation $R_v$ about $v\in V$, $(v,v)=-2$  
 acts as 
$$
R_v(u)=u'=-u-(u,v)v.
$$  
Given three points $v_1,v_2,v_3$, we can generate a group $G=\langle R_{v_1},R_{v_2},R_{v_3}\rangle$ acting on $\H^2$.

\subsubsection{Mutation.}
The initial matrix $B$ corresponds to the initial set of generating rotations in the group $G$ and to the initial triple of points in $\H^2$.
Applying mutations, we will obtain other sets of generating rotations of $G$ as well as other triples of points.

More precisely, define mutation of the set of generating rotations by partial conjugation, in exactly the same way as for reflections:
$$
\mu_k(r_j)=\begin{cases}  
r_kr_jr_k & \text{if  $b_{jk}>0$,}\\ 
r_j&  \text{otherwise.}\end{cases}
$$
Consequently, the mutation of the triple of points is defined by partial rotation:
$$
\mu_k(v_j)=\begin{cases}  
-v_j-(v_j,v_k)v_k & \text{if  $b_{jk}>0$,}\\ 
v_i&  \text{otherwise.}\end{cases}
$$

\begin{example} In Fig.~\ref{ex-rot} we show how a triple of points changes under mutation.
\begin{figure}[!h]
\begin{center}
\psfrag{dp}{\scriptsize $d_{kj}$}
\psfrag{dp'}{\scriptsize $d_{kj'}=d_{kj}$}
\psfrag{dq}{\scriptsize $d_{ik}$}
\psfrag{dr}{\scriptsize $d_{ij}$}
\psfrag{dr'}{\scriptsize $d_{ij'}$}
\psfrag{i}{\scriptsize $v_i$}
\psfrag{j}{\scriptsize $v_j$}
\psfrag{k}{\scriptsize $v_k$}
\psfrag{j'}{\scriptsize $v_{j'}$}
\epsfig{file=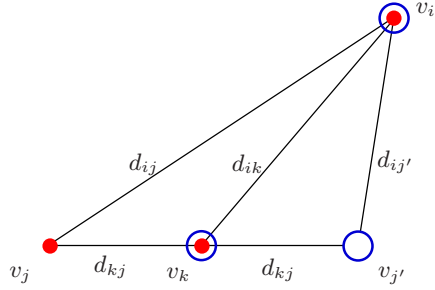,width=0.35\linewidth}
\caption{Mutation of a triple of points. Distances $d_{\alpha\beta}$ between points $v_\alpha$ and $v_\beta$ are equal to 
$\arccosh\frac{|b_{\alpha\beta}|}{2}$.}
\label{ex-rot}
\end{center}
\end{figure}

\end{example}

\subsection{Geometric realization by $\pi$-rotations}

\begin{lemma}
\label{rot}
Let $Q$ be a cyclic quiver of rank $3$ with all weights greater or equal to $2$, let $B$ be the corresponding skew-symmetric matrix with $b_{12},b_{23},b_{31}>0$, and let $V$ be the corresponding quadratic space. Suppose that $v_1,v_2,v_3\in V$ are vectors satisfying
\begin{itemize}
\item[(1)] $(v_i,v_i)=-2$,
\item[(2)] $(v_i,v_j)=-|b_{ij}|$ for $1\le i<j\le 3$;
\end{itemize}
Then $Q'=\mu_k(Q)$ is a cyclic quiver with weights  greater or equal to $2$, and the set of vectors ${\bf v'}=(\mu_k(v_1),\mu_k(v_2),\mu_k(v_3))$ satisfies conditions (1)-(2) for $B'=\mu_k(B)$.
\end{lemma}

\begin{proof}
Due to the symmetry, to prove the lemma  we only need to check one mutation (say, $\mu_2$).
As $\{v_1',v_2',v_3'\}=\{-v_1-(v_1,v_2)v_2,v_2,v_3\}$, we have
\begin{eqnarray*}
(v_1',v_3')&=&-(v_1,v_3)-(v_1,v_2)(v_2,v_3)=-(b_{13})-(-b_{12})(-b_{23})=-(b_{12}b_{23}-b_{31})=-b_{13}',\\
(v_1',v_2')&=&(v_1,v_2)=b_{12}',\\
(v_2',v_3')&=&(v_2,v_3)=b_{23}'.
\end{eqnarray*}
As $v_1'$ and $v_3'$ are negative vectors, $(v_1',v_3')=-2\cosh d(v_1',v_3')<-2<0$, which implies that $b_{31}'=-b_{13}'<-2$, i.e. $Q'=\mu_2(Q)$ is a cyclic quiver with  $|b'_{12}|,|b'_{23}|,|b'_{31}|\ge 2$ for $B'=\mu_2(B)$. 
Also, the computation above shows that  conditions (1)--(2) are satisfied by $\bf v'$ and $B'$.

\end{proof}

From now on, given a cyclic quiver we denote its weights by $p=|b_{12}|$, $q=|b_{23}|$, $r=|b_{31}|$.
We will also denote the corresponding matrix $B$ by a triple $(p,q,r)$.

\begin{lemma}
\label{r<2}
Let $Q$ be a cyclic quiver with weights $p,q,r>0$.
\begin{itemize}
\item[(a)] if $r<2$ then $Q$ is mutation-acyclic;
\item[(b)] if $r=2$ and $p\ne q$ then $Q$ is mutation-acyclic;
\item[(c)] if $r=2$  and $p=q\ge 2$ then $Q$ is mutation-cyclic, moreover $Q$ is minimal in the mutation class. 
\end{itemize}

\end{lemma}

\begin{proof}
(a) We will apply mutations $\mu_1$ and $\mu_3$ alternately (starting from $\mu_3$), so that at every step $b_{13}=r$ stays intact.
Furthermore, each of the steps changes either $b_{12}$ or $b_{23}$ in the following way: 

\medskip

\noindent
{\bf Claim $1$.} {\it 
For $n\in\N$ denote $Q_n'=(\mu_1\mu_3)^{n/2}Q$ if $n$ is even or $Q_n'=\mu_3(\mu_1\mu_3)^{(n-1)/2}Q$ if $n$ is odd. If all $Q_k'$ are cyclic for $k<n$, then the entries of the corresponding matrix $B_n'$ satisfy
$$
|b'_{12}|\  (\text{or } |b'_{23}|)= f_n(p,q,r)=u_{n}(r)q-u_{n-1}(r)p,  
$$
where $u_n(x)$ is a Chebyshev polynomial of the second kind (of a half-argument) recursively defined by 
$$ u_0(x)=1, \qquad u_1(x)=x \qquad u_{n+1}(x)=xu_n(x)-u_{n-1}(x).$$ }

\medskip
\noindent
Proof of Claim~$1$ is a straightforward induction:  
the base is   $\mu_3(p,q,r)=(rq-p,q,r)$; the step is given by $\mu=\mu_1$ or $\mu_3$ with
$$\mu(f_n,f_{n+1},r)=(f_{n+2},f_{n+1},r). $$
The claim can also be easily extracted from~\cite[Lemma 3.2]{LS}.



\medskip

\noindent
{\bf Claim $2$.} {\it For any real $p,q,r> 0$ s.t. $r<2$ there exists an integer $n>0$ such that  $u_{n+1}(r)q-u_n(r)p<0$.}

\medskip

To prove the claim, we will use Chebyshev polynomials of the second kind defined by 
$$ U_0(y)=1, \qquad U_1(y)=2y, \qquad U_{n+1}(y)=2yU_n(y)-U_{n-1}(y).$$
Notice that if $x=2y$ then $u_n(x)=U_n(y)$. For $0< r <2$ we can write $r=2\cos \theta$ for some $0< \theta<\pi/2$. Then we have
$$
u_n(r)=U_n(\cos \theta)=\frac{\sin((n+1)\theta)}{\sin \theta},
$$
where the last equality is a well-known property of Chebyshev polynomials of the second kind. If  $u_{n+1}(r)q-u_n(r)p\ge 0$, then 
$$\frac{\sin((n+1)\theta)}{\sin \theta}q\ge \frac{\sin(n\theta)}{\sin \theta}p,
$$
or just $\sin((n+1)\theta)q \ge \sin(n\theta)p$, as $\sin \theta>0$. Since  $0< \theta<\pi/2$, there exists an integer $n>0$ such that 
$\sin(k\theta)>0$ for all $0<k\le n$ but $\sin((n+1)\theta)<0$. This gives the number $n$ required in Claim 2.

\smallskip
Combining the two claims we see that there exists $n\in \N$ such that $Q_n'$ is acyclic, which completes the proof of part (a).

\medskip
\noindent
(b) If $r=2$ then $u_n(r)=n+1$, so, the condition  $u_{n+1}(r)q-u_n(r)p>0$ turns into $(n+1)q-np>0$.
Assuming $q<p$, this cannot hold if $n$ is large enough.

\medskip
\noindent
(c) If $p=q>2$ and $r=2$ then there exist points $v_1,v_2,v_3$ in $\H^2$ realizing $B=(q,q,r)$.
Indeed, we take $v_1=v_3$, and choose any $v_2$ such that $2\cosh d(v_1,v_2)=q$ (as usual, we assume $(v_i,v_i)=-2$).
Applying repeatedly Lemma~\ref{rot} we see that in this case $Q$ is mutation-cyclic. 
Moreover, the mutated triple of points  always remains collinear,
and it is easy to see that every new mutation either increases the distances in the triple or brings it to the previous configuration.
This implies that the initial quiver $Q$ was minimal.  

\end{proof}

Similarly to geometric realizations by reflections (see Definition~\ref{def-realization-refl}) we define 
 geometric realizations by $\pi$-rotations: 

\begin{definition}
\label{def-realization-rot}
Let $B$ be a $3\times 3$ skew-symmetric matrix.
We say that a triple of vectors ${\bf v}=(v_1,v_2,v_3)$ is a {\it geometric realization by $\pi$-rotations} of $B$ if conditions (1)--(2)
of Lemma~\ref{rot} are satisfied. We also say that $\bf v$ provides a {\it realization} of the mutation class of $B$ if the mutations of 
$\bf v$ via partial $\pi$-rotations agree with all the mutations of $B$, i.e. if conditions (1)--(2) are satisfied in every seed.

\end{definition}

We can now formulate the following immediate corollary of Lemma~\ref{rot}.

\begin{lemma}
\label{rot-no-ac}
A mutation-acyclic quiver has no realization by $\pi$-rotations.
\end{lemma}



\begin{theorem}
\label{existence}
Let $Q$ be a mutation-cyclic rank $3$ quiver, and let $B$ be the corresponding skew-symmetric matrix. Then the mutation class of $B$ has either a realization by reflections or a realization by $\pi$-rotations.

\end{theorem}

\begin{proof}
Since $Q$ is mutation-cyclic, Lemma~\ref{r<2} implies that $B=(p,q,r)$ with $p,q,r\ge 2$.
If there is a triple of points on $\H^2$ on mutual distances $d_p, d_q, d_r\ge 0$, where 
$$d_x=\arccosh \frac x 2,$$
then Lemma~\ref{rot} guarantees the realization by $\pi$-rotations
(as $2\cosh d(u,v) = -(u,v)$).
Such a triple of points on $\H^2$ does exist if and only if the triangle inequality holds for $d_p,d_q,d_r$.
More precisely, assuming $p\le q\le r$ (and hence  $d_p\le d_q\le d_r$), a hyperbolic triangle with sides   $d_p,d_q,d_r$ exists 
if and only if $$d_r\le d_p+d_q.$$

Now, suppose that $d_r>d_p+d_q$ (i.e. we are unable to construct a realization by $\pi$-rotations).
Notice, that we can also assume $p,q,r\ne 2$, as in the case $r=2$, $p=q\ge 2$ there is a realization by $\pi$-rotations 
(see the proof of Lemma~\ref{r<2}(c)) and in the case $r=2$, $p\ne q$ the matrix $B$ is mutation-acyclic by    Lemma~\ref{r<2}(b).
We will show that given  $d_p,d_q,d_r>0$ with $d_r>d_p+d_q$  there are three lines in $\H^2$ on mutual distances  $d_p,d_q,d_r$. Then choosing the normal vectors to these lines 
will lead to a geometric realization of $B$ by reflections (with respect to these lines): indeed, conditions (1) and (2) of 
Lemma~\ref{refl} will hold by construction and condition (3) will be easy to check (for instance, if we choose the directions of normal vectors as in Fig.~\ref{three lines}(a)
then all three scalar products will be positive).

\begin{figure}[!h]
\begin{center}
\psfrag{lp}{\small \color{blue} $l_p$}
\psfrag{lq}{\small \color{blue} $l_q$}
\psfrag{lr}{\small \color{blue} $l_r$}
\psfrag{lp'}{\small \color{blue}  $l_{p}'$}
\psfrag{dp}{\scriptsize \color{dred} $d_p$}
\psfrag{dq}{\scriptsize  \color{dred}  $d_q$}
\psfrag{dr}{\scriptsize \color{dred}   $d_r$}
\psfrag{a}{\small (a)}
\psfrag{b}{\small (b)}
\epsfig{file=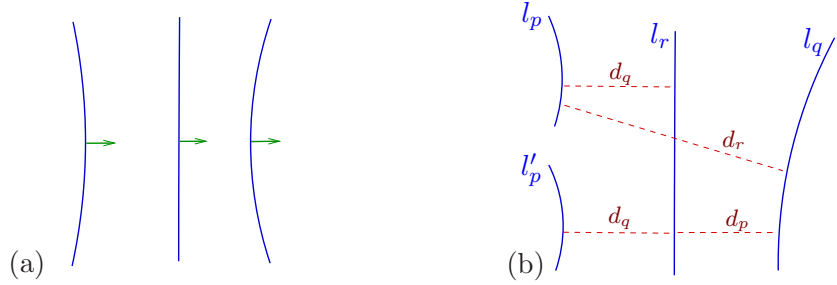,width=0.69\linewidth}
\caption{Three lines $l_p,l_q,l_r$ at distances $d_p,d_q,d_r>0$, where $d_r>d_p+d_q$.}
\label{three lines}
\end{center}
\end{figure}

It remains to show that for any  $d_p, d_q, d_r> 0$ with  $d_r>d_p+d_q$ there exists a triple of lines $l_p,l_q,l_r$ in $\H^2$ such that 
$d_p=d(l_q,l_r), d_q=d(l_p,l_r)$ and $d_r=d(l_q,l_p)$. To show this, we first choose two lines $l_r$ and $l_q$ on distance $d_p$ (see Fig~\ref{three lines}).
Also, we choose a line $l_p'$ so that  $d_q= d(l_r,l'_p)$, the lines $l_p',l_r,l_q$ have a common perpendicular and  $l_r$ separates $l_q$ from $l_p'$. When the line $l_p'$ slides along $l_r$ staying on the same distance (in other words, when we apply a hyperbolic translation along $l_r$), the distance between $l_p'$ and $l_q$ grows from $d_p+d_q$ to infinity. So, there is a position  $l_p$ of $l'_p$  for which  $d_r=d(l_q,l_p)$.

\end{proof}

\begin{remark}
It will be shown in the next section (Theorem~\ref{uniqueness}(2)) that cyclic quivers $(p,q,r)$ with  $d_r>d_p+d_q$ are in fact mutation-acyclic, so an existence of a realization by reflections is not a coincidence. However, for this we need some results from~\cite{BBH}.  

\end{remark}

\section{Geometry governed by the Markov constant}
\label{Markov-sec}

We have seen that every skew-symmetric rank $3$ real mutation class admits a geometric realization. In this section we study the geometric realizations obtained and show that their properties are controlled by the Markov constant  (see definition~\ref{Markov}).
We also show that all acyclic mutation classes are realized by reflections, all cyclic mutation classes are realized by $\pi$-rotations and 
both realizations may occur only for some degenerate cases (see Theorem~\ref{uniqueness}).

\begin{definition}
\label{Markov}
The {\it Markov constant} $C(p,q,r)$ for a triple $(p,q,r)$, where $p,q,r\in \R$ was introduced by Beineke, Br\"ustle, Hille in~\cite{BBH} as
$$ C(p,q,r)=p^2+q^2+r^2-pqr.
$$
\end{definition}

For a cyclic quiver $Q$ with weights $p,q,r$, $C(Q)$ is defined as $C(p,q,r)$ while for an acyclic quiver with weights  $p,q,r$ one has 
$C(Q):=C(p,q,-r)$ (this may be understood as turning an acyclic quiver into a cycle at the price of having a negative weight).
It is observed in~\cite{BBH} that $C(Q)$ is a mutation invariant, i.e. it is constant on the mutation class of $Q$.
It was also shown in~\cite{BBH} that in the case of integer weights $C(Q)$ characterizes (with some exceptions) the mutation-acyclic quivers:

\begin{prop}[\cite{BBH}, Theorem 1.2]
Let $Q$ be a cyclic quiver with integer weights given by $p,q,r\in \Z_{\ge 0}$. Then the following conditions are equivalent.
\begin{itemize}
\item[(1)] $Q$ is mutation-cyclic;
\item[(2)] $p,q,r\ge 2$ and $C(p,q,r)\le 4$;
\item[(3)] $C(p,q,r)<0$ or $Q$ is mutation-equivalent to one of the following classes:
\begin{itemize}
\item[(a)] $C(p,q,r)=0$, $(p,q,r)$ is mutation-equivalent to $(3,3,3)$;
\item[(b)] $C(p,q,r)=4$, $(p,q,r)$ is mutation-equivalent to $(q,q,2)$ for some $q>2$.
\end{itemize}
\end{itemize}

\end{prop}

\begin{remark}
Theorem~1.2 in~\cite{BBH} contains two more equivalent conditions, we are not reproducing them here.
\end{remark}

Our next aim is to give a geometric interpretation of $C(Q)$ (in particular, to explain why $C(Q)$ distinguishes mutation-acyclic quivers)
as well as to extend the result to the case of real numbers $p,q,r$.

The question of recognizing whether a quiver $Q$ is mutation-acyclic is non-trivial if $Q$ is not acyclic itself (i.e. $Q$ is a cycle $(p,q,r)$) and if $p,q,r\ge 2$ (otherwise we just use Lemma~\ref{r<2}(a)).
For quivers of this type, the proof of Theorem~\ref{existence} shows that $Q$ can be realized by $\pi$-rotations
(and is mutation-cyclic by Lemma~\ref{rot-no-ac}) or by reflections
depending on the triangle inequality for $d_r\le d_p+d_q$, where $p\le q\le r$ and $d_x=\arccosh \frac x 2$.
Denote  $$\Delta(Q)=d_p+d_q-d_r,$$
understanding $\Delta(Q)\ge 0$ as ``triangle inequality holds'' and  $\Delta(Q)< 0$ as ``it does not''.

\begin{lemma}
Let $Q=(p,q,r)$ be a rank $3$ cyclic quiver with $2\le p \le q \le r$.
Then 
\begin{itemize}
\item[-] if $\Delta(Q)>0$ then $C(Q)<4$;
\item[-] if $\Delta(Q)=0$ then $C(Q)=4$;
\item[-] if $\Delta(Q)<0$ then $C(Q)>4$.

\end{itemize}

\end{lemma}

\begin{proof}
$\Delta(Q)<0$ if and only if $\cosh(d_p+d_q)<\cosh(d_r)$. Here $\cosh(d_r)=r/2$ (as $d_x=\arccosh (x/2))$, which implies
$$
\cosh(d_p+d_q)=\frac{p}{2}\frac{q}{2}+\sinh(\arccosh \frac{p}{2})\sinh(\arccosh \frac{q}{2})
=\frac{pq}{4}+\sqrt{(\frac{p^2}{4}-1)(\frac{q^2}{4}-1)}.
$$
Hence, $\Delta(Q)<0$  is equivalent to $\sqrt{(p^2-4)(q^2-4)}<2r-pq$. Therefore, $\Delta(Q)<0$ implies $(p^2-4)(q^2-4)<(2r-pq)^2$, i.e. $4<p^2+q^2+r^2-pqr=C(Q)$. A straightforward calculation shows that $C(Q)>4$ and $2\le p \le q \le r$ imply $2r-pq>0$, so $C(Q)>4$ also implies $\Delta(Q)<0$. 

\end{proof}

\begin{theorem}
\label{uniqueness}
Let $Q$ be a rank $3$ quiver with real weights. Then
\begin{itemize}
\item[(1)] if $Q$ is mutation-acyclic then $C(Q)\ge 0$ and  $Q$ admits a realization by reflections; 
\item[(2)] if $Q$ is mutation-cyclic then $C(Q)\le 4$ and  $Q$ admits a realization by $\pi$-rotations; 
\item[(3)] $Q$ admits both realizations (by reflections and by $\pi$-rotations) if and only if
$Q$ is cyclic with $p,q,r\ge 2$ and $C(Q)=4$.

\end{itemize}
 
\end{theorem}

\begin{proof}
(1) If $Q$ is mutation-acyclic, consider the acyclic representative (we may assume it is $Q$ itself). Then $C(Q)\ge 0$ as it is a sum of four non-negative terms. Existence of a realization by reflections is guaranteed by Corollary~\ref{cor: ac by relf}.

\medskip
\noindent
(2) If $Q=(p,q,r)$ is mutation-cyclic, then by  Lemma~\ref{r<2}(a) we have $p,q,r\ge 2$ and by  Theorem~\ref{existence}, $Q$ has a realization either by reflections in $\H^2$ or by $\pi$-rotations (again, in $\H^2$). Which of the options holds depends on the triangle inequality,
or, in other words, on the sign of $\Delta(Q)$, which in its turn is determined by the sign of $C(Q)-4$.  
More precisely, if $C(Q)\le 4$ then the triangle inequality holds and $Q$ has a realization by $\pi$-rotations, and if $C(Q)>4$ then $Q$ has a realization by reflections.

Suppose that a mutation-cyclic quiver $Q$ has $C(Q)>4$ and, hence, has a realization by reflections.  It is shown in Section~5 of~\cite{BBH} that every mutation-cyclic class with $C(Q)\ne 4$ contains a minimal element $Q_{\mathrm{min}}$, where the sum of the weights $p+q+r$ is minimal over the whole mutation class (notice that \cite{BBH} shows this not only for integers but also for all mutation classes with real weights). Consider the realization of $Q_{\mathrm{min}}=(p_{\mathrm{min}},q_{\mathrm{min}},r_{\mathrm{min}})$. As $Q_{\mathrm{min}}$ is still mutation-cyclic, we have  
$p_{\mathrm{min}},q_{\mathrm{min}},r_{\mathrm{min}}\ge 2$ which implies that the lines $l_p,l_q,l_r$ in the realization of $Q_{\mathrm{min}}$ do not intersect each other. 
If one of the lines (say, $l_r$) separates the others (see Fig.~\ref{unique}(a)), then partial reflection in $l_r$ (reflection of exactly one of   $l_p$ and $l_q$) decreases one of the three distances, which contradicts the assumption that $Q_{\mathrm{min}}$ is minimal in the mutation class.   
If none of these lines separates the other two  (see Fig.~\ref{unique}(b)), then for any choice of normal vectors to these lines there will be even number of positive scalar products  $(v_i,v_j)$ (in particular, if we take all normals to be outward with respect to the triangular domain, then all three scalar products are negative). This does not agree with Definition~\ref{def-realization-refl} for a cyclic quiver.

In view of Theorem~\ref{existence}, the contradiction shows that every mutation-cyclic quiver $Q$ has $C(Q)\le 4$, admits a realization by $\pi$-rotations, and does not admit a realization by reflections if $C(Q)\ne 4$. 

\begin{figure}[!h]
\begin{center}
\psfrag{lp}{\scriptsize \color{blue}$l_p$}
\psfrag{lq}{\scriptsize  \color{blue}$l_q$}
\psfrag{lr}{\scriptsize  \color{blue}$l_r$}
\psfrag{dp}{\scriptsize  \color{dred} $d_p$}
\psfrag{dq}{\scriptsize \color{dred}  $d_q$}
\psfrag{dr}{\scriptsize \color{dred}  $d_r$}
\psfrag{a}{\small (a)}
\psfrag{b}{\small (b)}
\epsfig{file=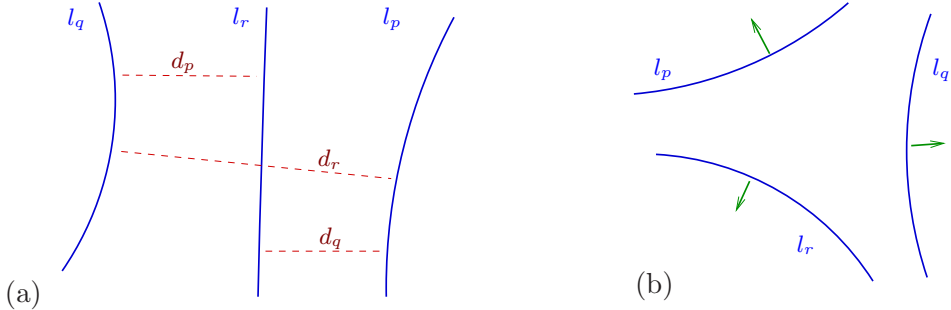,width=0.79\linewidth}
\caption{No realization by reflections for mutation-cyclic quivers.}
\label{unique}
\end{center}
\end{figure}

\medskip
\noindent
(3) First, by Lemma~\ref{rot-no-ac} a mutation-acyclic quiver cannot be realized by $\pi$-rotations.
Next, a mutation-cyclic quiver with $C(Q)\ne 4$ cannot be realized by reflections as shown in the proof of part (2).
Finally, suppose that $Q$ is mutation-cyclic and $C(Q)=4$. Then there is a realization of $Q$ by $\pi$-rotations about $3$ collinear points (as  $C(Q)=4$ is equivalent to the equality in the triangle inequality). Now, consider the line $l$  containing these three points. Taking three lines through these three points orthogonal to $l$ gives a realization by reflections. 

\end{proof}

\begin{remark}[On realizations of $(2,2,2)$]
In case of the quiver $(2,2,2)$ both realizations considered above turn out to be very degenerate (i.e. these are either reflections with respect to three coinciding lines or $\pi$-rotations with respect to three coinciding points).
However, one can also consider a realization by reflections with respect to three mutually parallel lines (in $\E^2$ or $\H^2$), this will lead to an infinite group $G$. 

\end{remark}

\begin{remark}
\label{rational-rem}
As it is mentioned in Section~5 of~\cite{BBH}, if $Q$ is  mutation-cyclic  with $C(Q)=4$ then the mutation class of $Q$ may have no minimal quiver. Having in mind any of the two realizations of $Q$ described above, it is clear that a mutation-cyclic $Q=(p,q,r)$ has a minimal representative in the mutation class if and only if $d_p/d_q\in \Q$. If $d_p/d_q\notin \Q$ then we can always make the distances between three collinear points (or between three lines) as small as we want, which means that the quiver tends to the Markov quiver $(2,2,2)$. 

\end{remark}

Theorem~\ref{uniqueness} describes the behavior of the realization and the value of $C(Q)$ based on the properties of the mutation class of $Q$ (mutation-cyclic versus mutation-acyclic). The following corollary discusses the properties of the realization based on the characteristics of an individual quiver $Q$.

\begin{cor} 
\label{cor unique}
Let $Q$ be a rank $3$ quiver.
\begin{itemize}
\item[(1)] If $Q$ is acyclic then $C(Q)\ge 0$ and there is a realization by reflections
\begin{itemize}
\item[-] in $\H^2$ if $C(Q)>4$;
\item[-] in $\E^2$ if $C(Q)=4$;
\item[-] in $\S^2$ if $C(Q)<4$.
\end{itemize}

\item[(2)] If $Q=(p,q,r)$ is cyclic with $\min(p,q,r)<2$  or with $r=2$, $p\ne q$, then 
$Q$ is mutation-acyclic,
$C(Q)\ge 0$ and there is a realization by reflections
\begin{itemize}
\item[-] in $\H^2$ if $C(Q)>4$;
\item[-] in $\E^2$ if $C(Q)=4$;
\item[-] in $\S^2$ if $C(Q)<4$.
\end{itemize}

\item[(3)] If $Q=(p,q,r)$ is cyclic with $\min(p,q,r)\ge 2$  then 
\begin{itemize}
\item[-] if $C(Q)>4$ then $Q$ is mutation-acyclic and there is a realization by reflections in $\H^2$; 
\item[-] if $C(Q)<4$ then $Q$ is mutation-cyclic and there is a realization by $\pi$-rotations;
\item[-] if $C(Q)=4$ then $Q$ is mutation-cyclic and there a both realizations. 
\end{itemize}

\end{itemize}

\end{cor}

\begin{proof}
In view of Theorem~\ref{uniqueness} and Lemma~\ref{r<2}(a) we only need to prove that $C(Q)$ is responsible for the choice of the space  
$\H^2$, $\E^2$ and $\S^2$ when $Q$ is mutation-acyclic.
The choice of this space depends on the sign of the determinant 
$$
\det\begin{pmatrix} 
2&-p&-q\\
-p&2&-r\\
-q&-r&2\\
\end{pmatrix} =-2(p^2+q^2+r^2+pqr-4)=-2(C(Q)-4),
$$
compare also to~\cite{S}.

\end{proof}

As we have seen above, $C(Q)$ is responsible for the choice of reflection/$\pi$-rotation realization (via triangle inequality). In case of a realization by reflections $C(Q)$ is also responsible for the choice of the space where the reflection group acts.
This suggests that in case of realization by $\pi$-rotations there should be also some geometric consequences of the value of $C(Q)$.
In Proposition~\ref{lambda} we show that the sign of $C(Q)$ characterizes the isometry of $\H^2$ obtained as a composition of three $\pi$-rotations with respect to the points corresponding to $Q$.

\begin{remark}[Types of isometries of $\H^2$]
There are three types of orientation-preserving isometries of hyperbolic plane:
\begin{itemize}
\item[-] {\it elliptic}, i.e. preserving one point inside $\H^2$ and rotating all other points about the fixed point by the same angle;
\item[-] {\it parabolic}, i.e. preserving exactly one point on the boundary  $\partial \H^2$; 
\item[-] {\it hyperbolic}, i.e. preserving exactly two points $X,Y\in \partial \H^2$ and moving all other points along the line $XY$.
\end{itemize}

Using an upper halfplane model of $\H^2$, one can associate an element $\A_g\in PSL(2,\R)$ to each orientation-preserving isometry $g$.
Then the type of the isometry $g$ depends on the square of the trace of $\A_g$ (note that $\tr(\A_g)$ is defined up to sign, so $\tr^2(\A_g)$ is well-defined):
\begin{itemize}
\item[-] if $\tr^2(\A_g)<4$ then  $g$ is elliptic (rotation by angle $\alpha$ where $2\cos \alpha=\tr(\A_g)$); 
\item[-] if $\tr^2(\A_g)=4$ then  $g$ is parabolic;
\item[-] if $\tr^2(\A_g)>4$ then  $g$ is hyperbolic (translation by $d$ where $2\cosh d=\tr(\A_g)$).

\end{itemize}

Details can be found in~\cite{B}.

\end{remark}

\begin{prop}
\label{lambda}
Let $Q=(p,q,r)$ be a mutation-cyclic quiver, let $A,B,C\in \H^2$ be the points providing a realization of $Q$ by $\pi$-rotations.
Let $R_A,R_B,R_C$ be the corresponding $\pi$-rotations, and let $g=R_A\circ R_B\circ R_C$. Then 
\begin{itemize}
\item[-] if $C(Q)>0$ then $g$ is elliptic (rotation by $\alpha$, where $2\cos \alpha=\sqrt{4-C}$);
\item[-] if $C(Q)=0$ then $g$ is parabolic;
\item[-] if $C(Q)<0$ then $g$ is hyperbolic (translation by $d$, where $2\cosh d= \sqrt{4-C}$).

\end{itemize} 

\end{prop}

\begin{proof}
The proof follows the ideas of~\cite[\textsection 11.5]{B}. To a triangle with sides $a,b,c$ and opposite angles $\alpha,\beta, \gamma$ one assigns a positive number $\lambda$ defined by
$$ \lambda:=\sinh a \sinh b \sin \gamma =\sinh b \sinh c \sin \alpha=\sinh c \sinh a \sin \beta,
$$
where the equality of these expressions follows from the (hyperbolic) sine rule.
Theorem~11.5.1 of~\cite{B} states that the trace of  $g=R_A\circ R_b\circ R_C$ equals $2\lambda$.
So, it is sufficient to prove that $2\lambda=\sqrt{4-C}$.

The proof of the latter statement is a short exercise in hyperbolic geometry using cosine law 
$$
\cos \gamma=\frac{\cosh c -\cosh a\cosh b}{\sinh a \sinh b} 
$$
and recalling that
$$
a=d_p=\arccosh \frac{p}{2}, \qquad \cosh a=\frac{p}{2}, \qquad \sinh a=\sqrt{\frac{p^2}{4}-1}.
$$

\end{proof}

\begin{remark}[Geometric meaning of $C(Q)$]
Theorem~\ref{uniqueness}, Corollary~\ref{cor unique} and Proposition~\ref{lambda} can be illustrated by Table~\ref{effects}.
We can see three geometric meanings of $C(Q)$:
\begin{itemize}
\item it tells whether $Q$ is mutation-acyclic or mutation-cyclic (i.e. admits realization by reflections or $\pi$-rotations);
\item for realization by reflections it chooses the space where the group $G$ acts;
\item for realization by rotations it tells the type of the product $g$ of three generators.

\end{itemize}

\end{remark}

\begin{table}[!h]
\begin{center}
\psfrag{S}{\small $\S^2, \ \ \ p,q,r<2$}
\psfrag{E}{\small $\E^2$}
\psfrag{H}{\small $\H^2$}
\psfrag{h}{\small $g$ is hyperbolic}
\psfrag{p}{\small $g$ parabolic}
\psfrag{e}{\small $g$ elliptic}
\psfrag{C}{ $C(Q)$}
\psfrag{0}{\color{red}  $0$}
\psfrag{4}{\color{red}  $4$}
\psfrag{a}{\small $Q$ is mutation-acyclic}
\psfrag{c}{\small $Q$ is mutation-cyclic, $p,q,r\ge 2$}
\epsfig{file=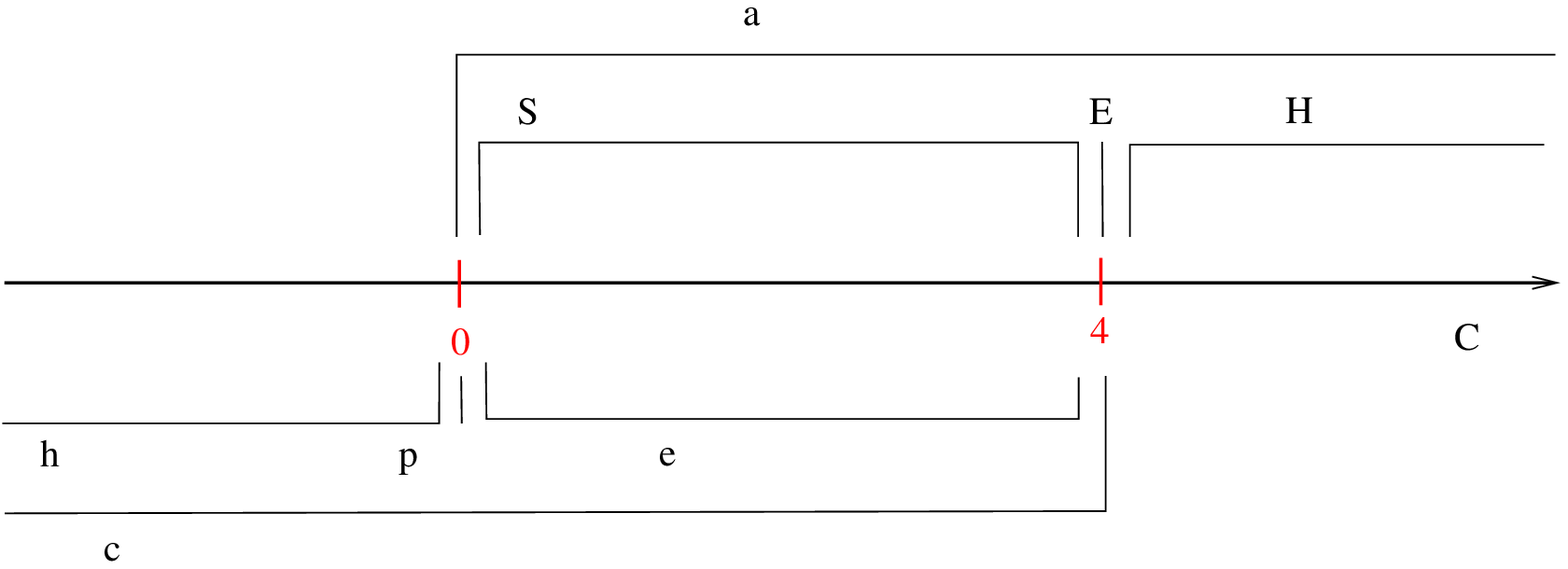,width=0.95\linewidth}
\caption{Geometric meanings of $C(Q)$.}
\label{effects}
\end{center}
\end{table}

\section{On discreteness of the groups realizing the mutation classes}
\label{s-discrete}

For each mutation-acyclic class we have constructed a group $G$ generated by three reflections,
and for each mutation-cyclic class we obtained a group $G$ generated by three $\pi$-rotations.
It is natural to ask whether the group acts in the corresponding space discretely. Here we collect the answers.

\subsection{Realizations by rotations} 
Let $G$ be a group generated by three $\pi$-rotations. 
It is shown in~\cite[Theorem~11.5.2]{B} that discreteness of $G$ is controlled by $\lambda$ defined in the proof of Proposition~\ref{lambda} (where we proved that $2\lambda=\sqrt{4-C(Q)}$). Reformulated in terms of $C(Q)$,  Theorem~11.5.2 of~\cite{B} implies that

\begin{itemize}
\item[(a)] if centers of the $\pi$-rotations are collinear, then 
\begin{itemize}
\item[ ] $G$ is discrete if and only if $d_p/d_q$ is rational (see Remark~\ref{rational-rem}).
\end{itemize}
\item[(b)] if centers of $\pi$-rotations are not collinear, then  
\begin{itemize}
\item[-] if $C(Q)<0$ then $G$ is discrete and has signature $(0:2,2,2;0;1)$;
\item[-] if $C(Q)=0$ then $G$ is discrete and has signature $(0:2,2,2;1;0)$;
\item[-] if $C(Q)>0$ then $G$ is discrete only if $\lambda=\frac{1}{2}\sqrt{4-C(Q)}$ takes one of the values:
$ \cos(\pi/k), k\ge 3; \qquad  \cos(2\pi/k), k\ge 5;  \qquad  \cos(3\pi/k), k\ge 7. 
$
\end{itemize}
\end{itemize}

\subsection{Realizations by reflections} 
Let $Q$ be a quiver admitting a realization via   a group $G$ generated by three reflections in $\S^2$, $\E^2$ or $\H^2$.
In view of Theorem~\ref{uniqueness} we can assume that $Q$ is mutation-acyclic (if $C(Q)=4$ then we are in assumptions of the case (a) above). Assume that $Q$ is acyclic itself  and denote $Q=(p,q,-r)$. Definition~\ref{def-realization-refl} (more precisely, property (3) of Lemma~\ref{refl}) implies that there is a choice of normals $v_1,v_2,v_3$ such that $(v_i,v_j)<0$, i.e. none of the three lines separates the other two, thus they bound some  domain $F$ in $\S^2$, $\E^2$ or $\H^2$ and this domain is {\it acute-angled} (i.e has no obtuse angles).

It is not always straightforward to see whether the group $G$ is discrete or not, but there are necessary conditions and there are sufficient conditions (see e.g. Example~9.8.5 in~\cite{B}):

\begin{itemize}
\item[(a)]{\bf  Necessary condition:} 
 If $G$ is discrete then either $p\ge 2$ or $p=2\cos{k\pi/l}$, where $k,l\in \Z_+$.
Similar conditions hold for $q$ and $r$.

\item[(b)]{\bf  Sufficient condition:} 
{  If each of $p,q,r$ either is greater or equal to $2$ or equals $2\cos{\pi/k}$ for some $k\in \Z_+$ (with possibly different $k$ for $p,q,r$),
then $G$ is discrete and $F$ is its fundamental domain.}

\item[(c)] { If necessary condition holds but the sufficient one does not} (i.e. all angles obtained are rational multiples of $\pi$, but not all of them are integer parts), then it is more involved to judge about the discreteness of $G$. Still,  based on~\cite{Cox,F,F1} one can see that:
\begin{itemize}
\item[-] { If $G$ is as above and acts on $\S^2$, then $G$ is discrete if and only if $(p,q,r)= (2\cos \pi t_1,2\cos \pi t_2,2\cos \pi t_3) $ where $(t_1,t_2,t_3)$ is one of 
$$
(\frac{1}{2},\frac{1}{3},\frac{2}{5}), \quad
(\frac{1}{3},\frac{1}{3},\frac{2}{5}), \quad 
(\frac{1}{2},\frac{1}{5},\frac{2}{5}), \quad 
(\frac{2}{5},\frac{2}{5},\frac{2}{5}).  
$$
}

\item[-] { If $G$ is as above and acts on $\E^2$, then $G$ is not discrete}\\
(as angles available in discrete reflection groups in $\E^2$ are integer multiples of $\pi/2, \pi/3, \pi/4$ and $\pi/6$ which do not produce acute angles of size $\pi k/l$ with coprime $k$ and $l$).

\item[-] { Suppose that $G$ is as above and acts on $\H^2$. If in addition $G$ is discrete and $F$ is tiled by finitely many fundamental domains of $G$,
then the tiling is as in Fig~\ref{tilings}.
To describe discrete groups with $F$ (of infinite volume) tiled by infinitely many fundamental domains one needs further investigations.

}

\begin{figure}[!h]
\begin{center}
\psfrag{a}{(1/2,1/k,2/k)}

\psfrag{(1/2,1/k,2/k)}{\scriptsize $(1/2,1/k,2/k)$}
\psfrag{(1/3,1/k,3/k)}{\scriptsize $(1/3,1/k,3/k)$}
\psfrag{(1/k,1/k,4/k)}{\scriptsize $(1/k,1/k,4/k)$}
\psfrag{(2/k,2/k,2/k)}{\scriptsize $(2/k,2/k,2/k)$}
\psfrag{(1/2,1/3,1/k)}{\scriptsize $(1/2,1/3,1/k)$}
\psfrag{(1/m,1/m,2/n)}{\scriptsize $(1/m,1/m,2/n)$}
\psfrag{(1/2,1/m,1/n)}{\scriptsize $(1/2,1/m,1/n)$}
\psfrag{(1/t,1/t,2/t)}{\scriptsize $(1/t,1/t,2/t)$}
\psfrag{(1/2,1/4,1/t)}{\scriptsize $(1/2,1/4,1/t)$}
\psfrag{(1/3,1/7,2/7)}{\scriptsize $(1/3,1/7,2/7)$}
\psfrag{(1/2,1/3,1/7)}{\scriptsize $(1/2,1/3,1/7)$}
\epsfig{file=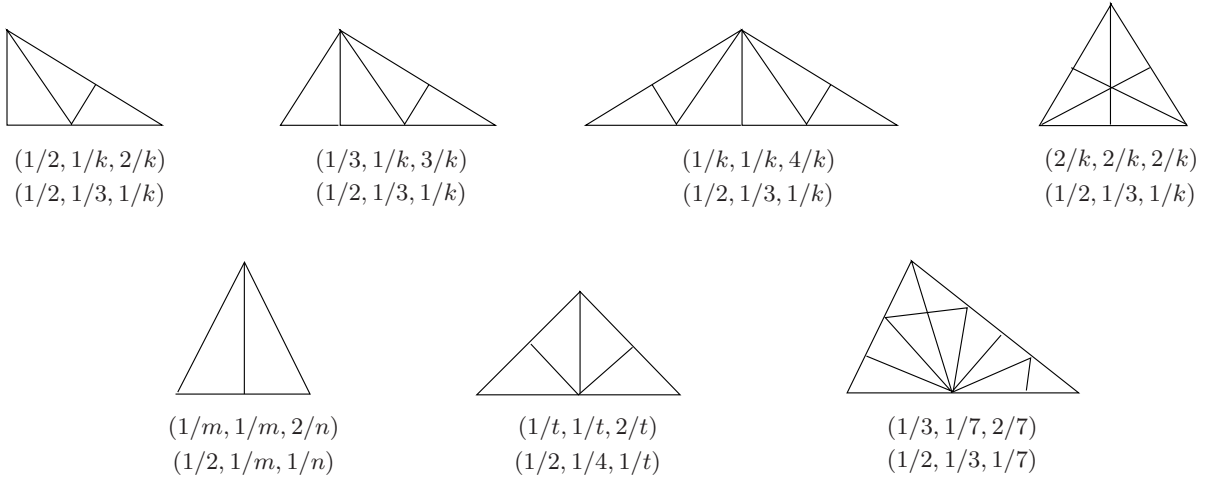,width=0.99\linewidth}
\caption{Tilings of hyperbolic triangles: the angles of the tiled triangle are written above the angles of the tiles. Here $k, m,n,t$ are integers satisfying  $k>6$, $1/m+1/n<1/2$, $t>4$, or they can be infinite. }
\label{tilings}
\end{center}
\end{figure}

\end{itemize}

\end{itemize}

\section{Classification of rank $3$ quivers of finite mutation type}
\label{s-fin}
In this section we use the geometric models constructed above to classify rank $3$ quivers with real weights having finite mutation classes.

The next lemma is obvious.

\begin{lemma}
If $Q=(p,q,\pm r)$ is mutation-finite quiver, then there is a minimal quiver in the mutation class (i.e. with minimal value of $p+q+r$).
Similarly, there is a maximal quiver in the mutation class.

\end{lemma}

\begin{lemma}
\label{2}
Let $Q=(p,q,r)$ or $Q=(p,q,-r)$ be a mutation-finite quiver, $p,q,r\in \R_{\ge 0}$. Then $p,q,r\le 2$. 

\end{lemma}

\begin{proof}
Suppose first that $Q$ is cyclic, i.e. $Q=(p,q,r)$, and assume   $p\ge q\ge r>0$. If $p> 2$, then $r'=pq-r>2q-r\ge q$, which implies that the mutation class contains an infinite sequence of quivers with strictly increasing sum of weights, so $Q$ cannot be mutation-finite.

Similarly, suppose $Q=(p,q,-r)$ (i.e. $Q$ is acyclic) with $\max(p,q,r)>2$. Applying if needed sink/source mutations, we may assume 
 $Q=(r,q,-p)$ with $p\ge q\ge r>0$, $p>2$. Then, after one more mutation we get a cyclic quiver with $p'=qr+p>2$ which results in an infinite mutation class as shown above.

\end{proof}

A combination of Lemma~\ref{r<2} with Lemma~\ref{2} leads to the following.
 
\begin{cor}
\label{cor 2}
If $Q$ is a mutation-cyclic quiver of finite mutation type then $Q=(2,2,2)$. 

\end{cor}



Corollary~\ref{cor 2} implies that we only need to consider mutation-acyclic quivers. By Theorem~\ref{uniqueness}(1)
every quiver of this type is represented by reflections in one of the spaces $\S^2,\E^2$ and $\H^2$ (depending on the sign of $4-C(Q)$).  





\begin{remark}[Notation] Given a quiver $Q=(p,q,\pm r)$ we will number its vertices so that 
$p=|b_{12}|$, $q=|b_{23}|$, $r=|b_{31}|$. 

\end{remark}

\begin{lemma}
\label{rational}
Suppose that $Q=(p,q,\pm r)$ is mutation-finite. Then $p= 2\cos(\pi k/l)$ for some $ k\in \Z_{\ge 0}, l\in \Z_+$. The same holds for $q$ and $r$.

\end{lemma}

\begin{proof}
By Lemma~\ref{2} we have $p,q,r\le 2$, so the lines $l_p$, $l_q$ and $l_r$ in the realization of $Q$ intersect each other forming some angles $\theta_p$,
 $\theta_q$, $\theta_r$ (if $p=2$ then the lines $l_q$ and $l_r$ are parallel, i.e. $\theta_p=0$). 

Suppose $Q=(p,q,\pm r)$ and  $p=2\cos \theta_p$. 
Applying $\mu_2$ and $\mu_1$ alternately, we will get infinitely many triples of lines $(l_p^{(n)},l_q^{(n)},l_r^{(n)})$ where $l_p=l_p^{(n)}$ and all lines 
$l_q^{(n)},l_r^{(n)}$ pass through the same point $O=l_q\cap l_r$ and form the same angle $\theta_p=\angle(l_q^{(n)},l_r^{(n)})=\angle(l_r^{(n)},l_q^{(n+1)})$,
see Fig.~\ref{irr}. If $\theta_p $ is not a rational multiple of $\pi$, then there are infinitely many intersection points of lines $l_r^{(n)}$ with $l_p$, which results in infinitely many different angles. Thus, the quivers obtained from $Q$ by alternating mutations  $\mu_2$ and $\mu_1$ will contain infinitely many different entries, which implies that $Q$ cannot be mutation-finite.

\end{proof}

\begin{figure}[!h]
\begin{center}
\psfrag{lp}{\scriptsize \color{blue} $l_p$}
\psfrag{lq}{\scriptsize \color{blue} $l_q$}
\psfrag{lr}{\scriptsize \color{blue} $l_r$}
\psfrag{lrn}{\scriptsize \color{dgreen} $l_r^{(n)}$}
\psfrag{lqn}{\scriptsize \color{dgreen} $l_q^{(n)}$}
\psfrag{tp}{\scriptsize  $\theta_p$}
\psfrag{tq}{\scriptsize  $\theta_q$}
\psfrag{tr}{\scriptsize  $\theta_r$}
\psfrag{O}{\scriptsize  $O$}
\epsfig{file=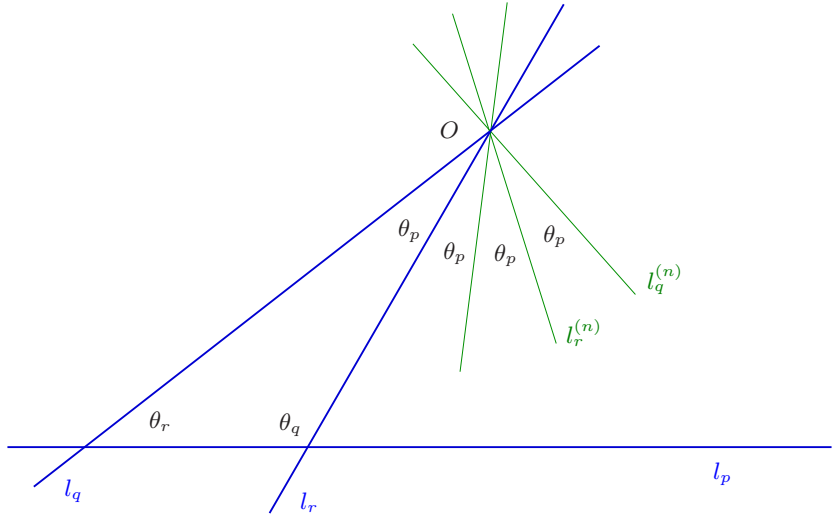,width=0.69\linewidth}
\caption{Angles are $\pi$-rational in mutation-finite case.}
\label{irr}
\end{center}
\end{figure}

\begin{lemma}
\label{H}
Let $Q$ be a mutation-acyclic quiver having a realization by reflections in $\H^2$ (i.e. $C(Q)>4$).
Then $Q$ is not mutation-finite.

\end{lemma}

\begin{proof}
By Lemma~\ref{2}, $p,q,r\le 2$, i.e. every quiver in the mutation class is represented by a triple $l_p,l_q,l_r$ of mutually intersecting (or parallel) lines.
First, suppose $p=2$ (i.e. $\theta_p=0$ and $l_q$ is parallel to $l_r$). 
By assumption $C(Q)>4$, which implies that 
$l_p,l_q,l_r$  are not mutually parallel. Hence, after several mutations preserving $l_p$ we will get a triple of lines $(l_p,l_q^{(n)},l_r^{(n)})$
 where $l_p$ is disjoint from $l_q^{(n)}$ and $l_r^{(n)}$, see Fig~\ref{no-hyp}(a). This contradicts  Lemma~\ref{2}.

\begin{figure}[!h]
\begin{center}
\psfrag{lp}{\scriptsize \color{blue} $l_p$}
\psfrag{lq}{\scriptsize \color{blue} $l_q$}
\psfrag{lr}{\scriptsize \color{blue} $l_r$}
\psfrag{lrn}{\scriptsize \color{dgreen} $l_r^{(n)}$}
\psfrag{lqn}{\scriptsize \color{dgreen} $l_q^{(n)}$}
\psfrag{tp}{\scriptsize  $\theta_p$}
\psfrag{tq}{\scriptsize  $\theta_q$}
\psfrag{tr}{\scriptsize  $\theta_r$}
\psfrag{tm}{\scriptsize  $\theta_{\mathrm{min}}$}
\psfrag{al}{\scriptsize  $\alpha$}
\psfrag{O}{\scriptsize  $O$}
\psfrag{s1}{\scriptsize  $S_1$}
\psfrag{s2}{\scriptsize  $S_2$}
\psfrag{a}{\small (a)}
\psfrag{b}{\small (b)}
\epsfig{file=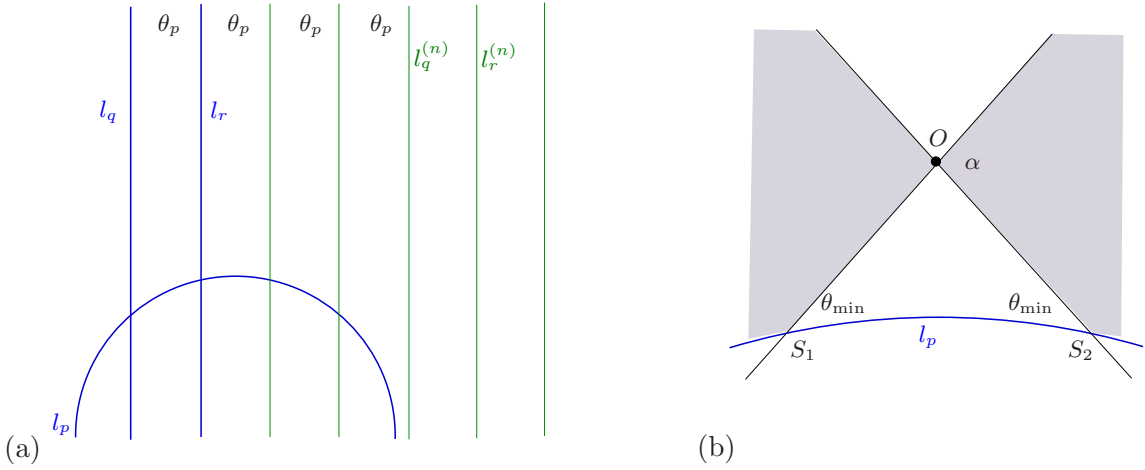,width=0.95\linewidth}
\caption{Hyperbolic case. (We use upper half-plane model on the left and Poincar\'e disc model on the right).}
\label{no-hyp}
\end{center}
\end{figure}

So, $Q$ (as well as every quiver in its mutation class) is realized by a triple of lines forming a compact triangle.
The angles $\theta_p, \theta_q, \theta_r$ in this triangle representing  $Q=(p,q,\pm r)$ are functions of  $p,q,r$: 
$p=2\cos \theta_p$ (and similar for $q$ and $r$). So, if $Q$ is mutation-finite then there is a smallest non-zero angle $\theta_{\mathrm{min}}$ such that $\theta_{\mathrm{min}}$ appears as an angle for a realization of some $Q'$ in the mutation class of $Q$.

Consider the realization $T_0=(l_p,l_q,l_r)$ of the quiver $Q'$ and let $\theta_{\mathrm{min}}=\angle(l_q,l_r)$. Applying alternately mutations  $\mu_2$ and $\mu_1$ 
as in the proof of Lemma~\ref{rational} (so that $l_p$ is always preserved and the image of $l_q$ is reflected with respect to the image of $l_r$ or vice versa), we will get further triangles $T_i$ realizing different quivers in the mutation class. Our aim is to show that either some of the triangles $T_i$ contains an angle smaller than $\theta_{\mathrm{min}}$ or some of $T_i$ has two disjoint sides (in contradiction with Lemma~\ref{2}).

Let $O=l_q\cap l_r$. Consider the lines through $O$ forming the angle $\theta_{\mathrm{min}}$ with $l_p$ (see Fig.~\ref{no-hyp}(b)), let $S_1$ and $S_2$ be the intersection points of these lines with $l_p$. Let $\alpha$ be the angle formed by these lines (see Fig.~\ref{no-hyp}(b)).
Each of the triangles $T_i$ has $O$ as a vertex, and as the sum of angles in a hyperbolic triangle is less than $\pi$, we have
$
\angle S_1OS_2<\pi - 2 \theta_{\mathrm{min}}, 
$
which implies that  $$\alpha=\pi-\angle S_1OS_2>2\theta_{\min}.$$
This means that at least one of the triangles $T_i$ will have a side crossing the grey domain between the lines. However, such a line will either be disjoint from $l_p$ or parallel to $l_p$ (contradicting Lemma~\ref{2} or the case considered above respectively), or it will cross $l_p$ at an angle smaller than $\theta_{\mathrm{min}}$ which is not possible either. The contradiction completes the proof of the lemma.

\end{proof}

\begin{lemma}
\label{E}
Suppose that $Q$ is mutation-acyclic and has a  realization by reflections in $\E^2$ (i.e. $C(Q)=4$).
Then the following conditions are equivalent:
\begin{itemize}
\item[(a)] $Q$ is mutation-finite;
\item[(b)] $Q=(p_1,p_2,\pm p_3)$ with $p_i=2\cos (\pi t_i)$, where $t_i\in \Q$;
\item[(c)] $Q$ is mutation-equivalent to  $(2\cos (\pi/n ),2\cos(\pi/n),2)$, where   $n\in \Z_+$.

\end{itemize}

\end{lemma}

\begin{proof}
Condition (a) implies (b) by Lemma~\ref{rational}. 
Next, (b) says that in the realization $(l_0,l_1,l_2)$, one has $\angle(l_1,l_0)=k_1\pi/n_1$ and $\angle(l_2,l_0)=k_2\pi/n_2$
for some $k_i,n_i\in \Z_+$. 
This implies that under the mutations one can only obtain angles of size $k \pi/n_1n_2$, where $k\in \Z_+$, $k<n_1n_2$. So, in any quiver mutation-equivalent to $Q$ the weights can only take finitely many values $2\cos (k \pi/n_1n_2)$, which results in finitely many quivers in the mutation class.
This shows equivalence of (a) and (b). Obviously, (c) implies (b). We are left to show that (c) follows from either (a) or (b).

Assume (a) holds, i.e. $Q$ is mutation-finite. Then there is a minimal angle $\theta_{\mathrm{min}}$ obtained as an angle between the lines 
in a realization of some quiver $Q'$ in the mutation class of $Q$.  Assume that $\theta_{\mathrm{min}}=\angle(l_1,l_2)$ and consider the alternating sequence of 
mutations $\mu_1$ and $\mu_2$.
Up to conjugation, we may assume that all these mutations preserve $l_0$ and reflect the image of $l_2$ with respect to the image of $l_1$ (or vice versa).
We obtain finitely many lines $l_1,l_2,\dots,l_m$  through $O=l_1\cap l_2$, any two adjacent lines $l_i$ and $l_{i+1}$ form an angle $\theta_{\min}$ and belong to a realization of one quiver (together with $l_0$). As the angle formed by $l_0$ and any  of these lines cannot be smaller than $\theta_{\mathrm{min}}$, we conclude that $\theta_{\mathrm{min}}=\pi/n$ for some integer $n$, and one of $l_1, \dots, l_m$, say $l_i$, is parallel to $l_0$, see Fig.~\ref{eucl}. Then the lines $(l_i,l_{i-1},l_0)$ form a realization of some quiver $Q''$ in the mutation class of $Q$, where $Q''=(2\cos \theta_{\mathrm{min}},2\cos\theta_{\mathrm{min}},2)$, $\theta_{\mathrm{min}}=\pi/n$. This shows that (a) implies (c). 


\end{proof}

\begin{figure}[!h]
\begin{center}
\psfrag{l0}{\scriptsize \color{blue} $l_0$}
\psfrag{li}{\scriptsize \color{blue} $l_i$}
\psfrag{li1}{\scriptsize \color{blue} $l_{i-1}$}
\psfrag{tm}{\scriptsize  $\theta_{\mathrm{min}}$}
\epsfig{file=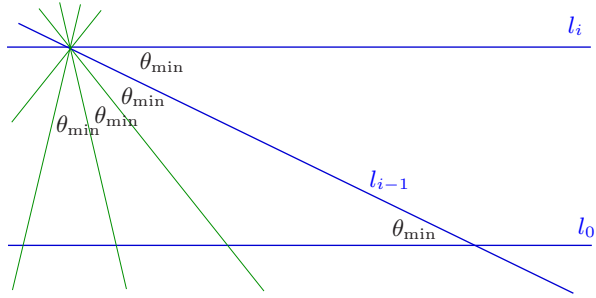,width=0.49\linewidth}
\caption{Angles are $\pi$-rational in mutation-finite case.}
\label{eucl}
\end{center}
\end{figure}

\begin{remark}
The acyclic representative in the mutation class of $(2\cos (\pi/n ),2\cos(\pi/n),2)$ is 
\begin{itemize}
\item $(2\cos\frac{\pi}{n} ,2\cos(\frac{\pi}{2}-\frac\pi {2n} ),2\cos(\frac\pi 2+\frac\pi {2n} )) $, if $n$ is odd;
\item $(2\cos\frac{\pi}{n} ,2\cos(\frac\pi 2-\frac\pi n ),0) $, if $n$ is even.
\end{itemize}

\end{remark}




\begin{lemma}
\label{S}
Suppose that $Q$ is mutation-acyclic and has a  realization by reflections in $\S^2$ (i.e. $C(Q)<4$).
If $Q$ is mutation-finite then $Q$ is mutation-equivalent to $$(2\cos (\pi t_1),2\cos(\pi t_2),0)$$ where $(t_1,t_2)$ is one of the following pairs:
$$(\frac{1}{3},\frac{1}{3}), \ \
(\frac{1}{3},\frac{1}{4}), \ \
(\frac{1}{3},\frac{1}{5}), \ \
(\frac{1}{3},\frac{2}{5}), \ \
(\frac{1}{5},\frac{2}{5}). 
$$

\end{lemma}

\begin{proof}
Recall from Lemma~\ref{rational} that the weights in mutation-finite quivers are of the form $2\cos \theta$, where $\theta$ is a rational multiple of $\pi$. 
To prove the lemma we will apply a mutation $\mu$ to $Q$ and check whether $\mu(Q)$ still satisfies this condition.

More precisely, we may assume that $Q$ is acyclic and 
$$Q=(p,q,-r)=(2\cos\frac{\pi t}{n},2\cos\frac{\pi s}{n}, -2\cos\frac{\pi m}{n} ),
$$
where $0<\frac{\pi t}{n}\le\frac{\pi s}{n}\le \frac{\pi m}{n}\le \frac{\pi}{2}$ and $n$ is an integer such that $\pi/n$ is the smallest angle in the realization of the mutation class of $Q$.
Applying the mutation preserving $p$ and $q$ and changing $r$ to $r'$ we get 
\begin{eqnarray*}
r'=pq+r&=&4\cos\frac{\pi t}{n}\cos\frac{\pi s}{n}+2\cos\frac{\pi m}{n}\\
       &=&2\cos\frac{\pi(s+t)}{n}+ 2\cos\frac{\pi(s-t)}{n}+2\cos\frac{\pi m}{n}.
\end{eqnarray*}
Notice that $r'$ should be also a double cosine of a rational multiple of $\pi$, or, more precisely, an integer multiple of $\pi/n$.
So, if $Q$ is mutation-finite, then there are integer numbers $s,t,m,k,n$  satisfying the equation
\begin{equation}
\label{1}
\cos\frac{\pi(s+t)}{n}+ \cos\frac{\pi(s-t)}{n}+\cos\frac{\pi m}{n}=\cos\frac{\pi k}{n}.
\end{equation}

Similar equations were considered by Conway and Jones in~\cite{CJ}. More precisely, it is shown in~\cite{CJ} that
given at most four rational multiples of $\pi$ lying between $0$ and $\pi/2$, and assuming that a rational linear combination of their cosines gives a rational number (but no proper subset has this property), this linear combination has to be one of the following:

\begin{center}
\begin{tabular}{l}
$\cos\pi/3=1/2$,\\
$-\cos\varphi +\cos(\pi/3-\varphi)+\cos(\pi/3+\varphi)=0$ \ \ ($0<\varphi<\pi/6$),\\
$\cos \pi/5-\cos2\pi/5=1/2$,\\
$\cos\pi/7-\cos2\pi/7+\cos3\pi/7=1/2$,\\
$\cos\pi/5-\cos\pi/15+\cos 4\pi/15=1/2$,\\
$-\cos2\pi/5+\cos2\pi/15-\cos 7\pi/15=1/2$,\\
\end{tabular}
\end{center}
or one of four other equations, each involving four cosines on the left and $1/2$ on the right. 

The latter four equations are irrelevant for us as they have too many terms to result in an equation of type~(\ref{1}). 
So, we need to consider the former six equations and a trivial identity  
$\cos\varphi+\cos \psi=\cos\varphi+\cos \psi$.
For each of these identities we match its terms to the terms of~(\ref{1}) (taking into account the signs of the terms, there can be up to 
six ways to do these matchings), and compute the values of $s,t,m,k,n$.
Most of the values obtained by this procedure are not relevant by one of the two reasons:
\begin{itemize}
\item[-] either the values $s,t,m,n$ correspond to a  triangle in $\H^2$ or $\E^2$, but not in $\S^2$ as needed;
\item[-] or the values $s,t,m,n$ do not correspond to an acute-angled triangle (which should be the case as we start with an acyclic quiver $Q$).
\end{itemize}

After removing irrelevant results, there are  $13$ cases left, some of them are still irrelevant (correspond to mutation-infinite quivers).
To exclude these, we check one more mutation and write an equation similar to~(\ref{1}) for $rq+p$ or $rp+q$.
Removing these, we result in five quivers listed in the lemma plus two more quivers:
$(1,1,-2\cos 2\pi/5)$ and $(2\cos 2\pi/5,2\cos 2\pi/5,-2\cos 2\pi/5)$,
which turned out to be mutation-equivalent to 
$(2\cos\pi/5,2\cos2\pi/5,0)$ and $(1,2\cos2\pi/5,0)$ respectively.


\end{proof}

\begin{remark}(Finite mutation classes, spherical case)
In Table~\ref{finite} we list the quivers belonging to the five finite mutation classes  described by Lemma~\ref{S}.
Notice that two of these classes contain two acyclic representatives which are not sink/source equivalent.

\end{remark}

\begin{table}[!h]
\begin{center}
\caption{Finite mutation classes with $C(Q)<4$}
\label{finite}
\begin{tabular}{|ll|c|}
\hline
&&\\
Acyclic quivers & Cyclic quivers&$C(Q)$\\
(up to sink/source)&&\\
\hline
&&\\
$(1,1,0)$& $ (1,1,1)$&2\\
&&\\
\hline
&&\\
$(1,\sqrt 2,0)$& $ (\sqrt 2,\sqrt 2,1)$&3\\
&&\\
\hline
&&\\
$(1,2\cos\pi/5,0)$& $ (2\cos\pi/5,2\cos\pi/5,1)$&$\frac{5+\sqrt 5}{2}$\\
&$ (2\cos\pi/5,2\cos\pi/5,2\cos\pi/5)$ &\\
&&\\
\hline
&&\\
$(2\cos\pi/5,2\cos 2\pi/5,0)$& $ (2\cos\pi/5,2\cos2\pi/5,1)$  &3\\
$(1,1,-2\cos 2\pi/5)$ & $ (1,1,2\cos\pi/5)$&\\
&&\\
\hline
&&\\
$(1,2\cos2\pi/5,0)$&$ (2\cos2\pi/5,2\cos2\pi/5,1)$&$\frac{5-\sqrt 5}{2}$\\
 $ (2\cos2\pi/5,2\cos 2\pi/5,-2\cos 2\pi/5)$ &&\\
&&\\
\hline

\end{tabular}
\end{center}
\end{table}



Corollary~\ref{2} together with Lemmas~\ref{H},~\ref{E} and~\ref{S} imply the following classification.

\begin{theorem}
\label{finite-thm}
Let $Q$ be a connected rank $3$ quiver with real weights. Then $Q$ is of finite mutation type if and only if it is mutation-equivalent
to one of the following quivers:
\begin{itemize}
\item[(1)] $(2,2,2)$;
\item[(2)]  $(2\cos (\pi/n ),2\cos(\pi/n),2)$,  $n\in \Z_+$;
\item[(3)] $(1,1,0)$, $(1,\sqrt 2,0)$, $(1,2\cos\pi/5,0)$, $(2\cos\pi/5,2\cos 2\pi/5,0)$, $(1,2\cos2\pi/5,0)$.
\end{itemize}

\end{theorem}

\begin{remark}
\label{finite type}
The five mutation classes in part (3) of Theorem~\ref{finite-thm} contain all rank $3$ quivers of ``finite type'', i.e. ones that can be modeled by reflections of finitely many vectors (see Remark~\ref{signed-mut} for the precise definition of mutation we use here). Namely, the first three correspond to types $A_3$, $B_3$ and $H_3$, the exchange graphs for these classes can be found in~\cite{FR}. The remaining two can also be modeled by reflections in some of the roots of the non-crystallographic root system $H_3$, we draw the corresponding ``exchange graphs'' in Fig.~\ref{exchange}.      
\end{remark}

\begin{remark}
\label{cover}
It is easy to see that the triangular domains corresponding to quivers in the mutation classes of $A_3$, $B_3$ and $H_3$ tessellate the $2$-sphere. The domains corresponding to quivers in the mutation classes of $(2\cos\pi/5,2\cos 2\pi/5,0)$ and $(1,2\cos2\pi/5,0)$ tessellate a torus which is a $2$- or $4$-fold covering of the sphere respectively.
 
\end{remark}

\begin{figure}
\psfrag{A}{\tiny  $a$}
\psfrag{B}{\tiny  $b$}
\psfrag{C}{\tiny  $c$}
\psfrag{D}{\tiny  $d$}
\psfrag{E}{\tiny  $e$}
\psfrag{F}{\tiny  $f$}
\psfrag{G}{\tiny  $g$}
\psfrag{H}{\tiny  $h$}
\psfrag{M}{\tiny  $m$}
\psfrag{N}{\tiny  $n$}
\epsfig{file=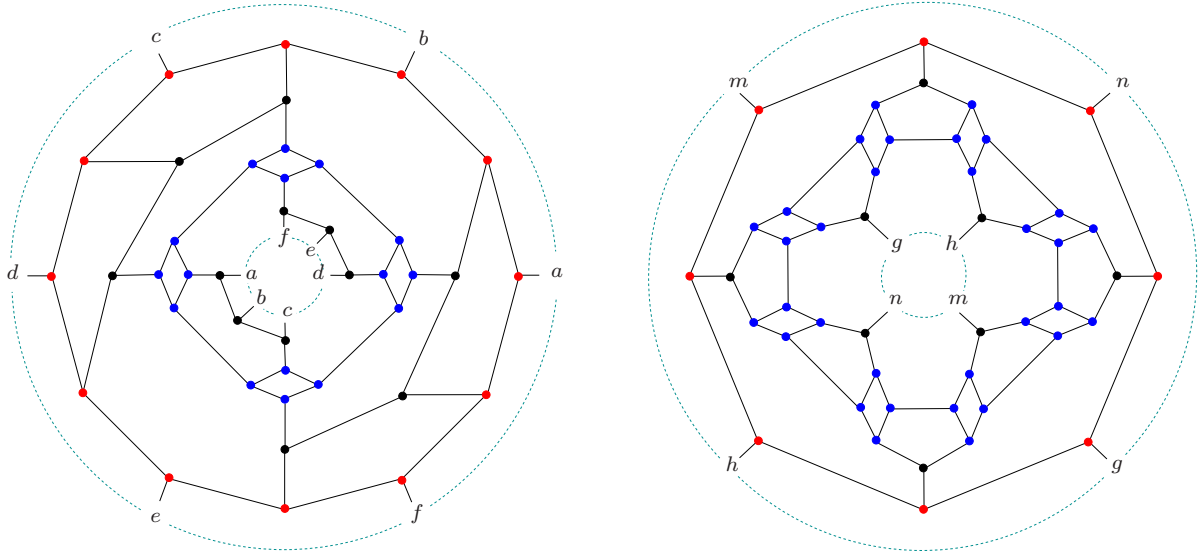,width=0.99\linewidth}
\caption{``Exchange graphs'' of the  mutation classes for quivers $(1,2\cos\pi/5,0)$ on the left and
$(2\cos\pi/5,2\cos 2\pi/5,0)$ on the right. Each graph contains two acyclic belts (blue and red) and can be drawn on a torus (by identifying the boundaries of annuli along the same letters).}
\label{exchange}
\end{figure}

\section{Acyclic representatives in infinite real mutation classes}
\label{s-ac}

Table~\ref{finite} shows that the structure of acyclic representatives in mutation classes of quivers with real weights is very different from the one we are used to see in the integer case. In particular, there may be acyclic representatives in the same mutation class which differ much more than just by a sequence of sink/source mutations.

\begin{lemma}
\label{line}
Let $Q=(p,q,r)$ be a mutation-acyclic quiver with $p<2$. 
Then iterating mutations $\mu_1$ and $\mu_2$ (so that $l_p$ and the weight $p$ are always preserved), one can always reach an acyclic representative in at most  $\lfloor\pi/\arccos\frac{p}{2}\rfloor$ mutations.

In particular, there is an acyclic representative containing weight $p$.
\end{lemma}

\begin{proof}
Consider a realization $(l_p,l_q,l_r)$ of $Q$ by reflections and consider the triples of lines obtained from  $(l_p,l_q,l_r)$
by mutations  $\mu_1$ and $\mu_2$ applied alternately (see Fig.~\ref{irr}). If $n$ consecutive sectors cover the whole angle $2\pi$ around the common point $O$ of $l_q$ and $l_r$, then at least one of the corresponding $\lfloor (n+1)/2\rfloor$ triples is acute-angled (one can draw the line through $O$ orthogonal to $l_p$ and take the two mutations of $l_q$ and $l_r$ composing smallest angles with it). 

Since $\arccos\frac{p}{2}=\theta_p\ge\frac{2\pi}{n}$, we can take $n$ to be equal to $\lfloor2\pi/\arccos\frac{p}{2}\rfloor+1$.
As one needs to make $\lfloor(n-1)/2\rfloor$ mutations to obtain all the $\lfloor(n+1)/2\rfloor$ triples that produce $n$ sectors covering $2\pi$, the number of required mutations is then does not exceed $\lfloor\pi/\arccos\frac{p}{2}\rfloor$. 
\end{proof}

\begin{theorem}
\label{acyclic in S}
Let $Q=(p,q,r)$ be a mutation-acyclic quiver with $0<C(Q)<4$.
Then there exists an acyclic quiver $Q'$ which can be obtained from $Q$ in at most
 $\!\lfloor\pi/\!\arcsin\!\frac{\sqrt{4-C(Q)}}{2}\!\rfloor$ mutations.

\end{theorem}

\begin{proof}
Consider the realization $(l_p,l_q,l_r)$ of $Q$ by reflections. As $C(Q)<4$, this realization is a configuration of $3$ lines on a sphere.
By Lemma~\ref{line}, it is sufficient to show the angles in the realization of other quivers in the mutation class cannot be too small.
We will show that they cannot become smaller than $\arcsin\frac{\sqrt{4-C(Q)}}{2}$.

To show this we follow the same ideas as in the proof of Lemma~\ref{lambda}.
Namely, we choose a triangle bounded by  $(l_p,l_q,l_r)$ and denote the lengths of its sides by $a,b,c$ and the opposite angles by $\alpha,\beta,\gamma$. 
Then we show that 
$$
\lambda:=\sin a \sin \beta \sin \gamma=\sin b \sin\alpha \sin \gamma =\sin c \sin \beta\sin \gamma
= \frac{1}{2}\sqrt{4-C(Q)}.
$$
Here all but the last equalities follow from the spherical sine law, and the last equality follows from spherical  second cosine law
$$
\cos a =\frac{\cos \beta \cos\gamma -\cos \alpha}{\sin \beta \sin \gamma}
$$
while taking in mind that 
$$
p=2\cos \alpha, \qquad q=2\cos\beta, \qquad r=2\cos \gamma. 
$$
In particular, we see that $$\sin\gamma\ge \sin a \sin \beta \sin \gamma =  \frac{1}{2}\sqrt{4-C(Q)}.$$
                 
As $C(Q)$ is independent on the representative in the mutation class, we have the same estimate for every angle in every triangle we can obtain by mutations of  $(l_p,l_q,l_r).$

\end{proof}

\begin{remark}
There is no counterpart of Theorem~\ref{acyclic in S} for the case of $C(Q)\ge 4$ (i.e. for Euclidean and hyperbolic realizations). 
Indeed, take any triple of lines  $(l_p,l_q,l_r)$ in $\E^2$, where $l_q$ and $l_r$ form a $\pi$-irrational angle $\theta_p$. Then one can use mutations $\mu_1$ and $\mu_2$ to obtain a triple of lines with (at least one) arbitrary small angle. Repeating the same but now centered in the smallest angle, we can get a triple of lines with two angles arbitrary small (and thus the third one arbitrary close to $\pi$), i.e. a triple of almost collinear lines. It is easy to see that this cannot be turned into an acute-angled configuration in a predefined number of mutations. 

\end{remark}

\end{document}